\documentclass[11pt]{article}

\usepackage[margin=1in]{geometry}
\usepackage{amsmath,amssymb,amsthm,mathtools}
\usepackage{bbm}
\usepackage{graphicx}
\usepackage[colorlinks=true,linkcolor=blue,citecolor=blue,urlcolor=blue]{hyperref}

\newtheorem{theorem}{Theorem}
\newtheorem{proposition}[theorem]{Proposition}
\newtheorem{lemma}[theorem]{Lemma}

\newtheorem{remark}[theorem]{Remark}
\newtheorem{definition}{Definition}

\newcommand{\R}{\mathbb R}
\newcommand{\E}{\mathbb E}
\newcommand{\Pp}{\mathbb P}
\newcommand{\tr}{\operatorname{tr}}
\newcommand{\Var}{\operatorname{Var}}
\newcommand{\diag}{\operatorname{diag}}
\newcommand{\op}{\operatorname{op}}
\newcommand{\one}{\mathbf 1}
\newcommand{\He}{\operatorname{He}}
\newcommand{\ip}[2]{\left\langle #1,#2\right\rangle}
\newcommand{\norm}[1]{\left\|#1\right\|}
\newcommand{\opnorm}[1]{\left\|#1\right\|_{\op}}
\newcommand{\fnorm}[1]{\left\|#1\right\|_F}
\newcommand{\stablerank}{r_{\mathrm{st}}}
\newcommand{\effectivedim}{d_{\mathrm{eff}}}

\title{Recovery of latent inner products from an anisotropic Gaussian random geometric graph}
\author{%
  Cheng Mao\thanks{School of Mathematics, Georgia Institute of Technology.
    \texttt{cheng.mao@math.gatech.edu}.}
  \and
  Vidya Muthukumar\thanks{School of Electrical and Computer Engineering and
    H.\ Milton Stewart School of Industrial and Systems Engineering,
    Georgia Institute of Technology. \texttt{vmuthukumar8@gatech.edu}.}
}
\date{\today}

\begin{document}
\maketitle

\begin{abstract}
We study the problem of recovering latent inner products from a random geometric graph with anisotropic Gaussian latent points. More precisely, for an i.i.d.\ sample $x_1, \dots, x_n \sim N(0,\Sigma)$ where $\Sigma \in \mathbb{R}^{d \times d}$, an edge $(i,j)$ is present in the graph if and only if $\langle x_i, x_j \rangle \ge \zeta$ for a threshold $\zeta$. We assume the threshold $\zeta$ to be chosen such that the average edge density of the graph is of constant order. To address the undesired degree fluctuations amplified by the anisotropy of the latent points, we consider the doubly centered adjacency matrix of the graph, and estimate the latent inner products using a rank-$d$ spectral approximation of the doubly centered matrix. The estimator obtains a mean squared error with a rate involving the stable rank of the covariance matrix $\Sigma$. Notably, the rate of estimation matches the state of the art for the isotropic case $\Sigma = I_d$, and permits an ill-conditioned covariance matrix with a diverging condition number. The analysis of the spectral method proceeds via the entrywise Hermite expansion of the doubly centered adjacency matrix with respect to the latent inner products. Instead of the standard trace method, it uses a decoupling argument recently introduced by Kaushik, Romberg, and Muthukumar (2025) to control nonlinear error terms.
\end{abstract}

\tableofcontents

\section{Introduction}
\label{sec:introduction}

Random geometric graphs are latent space models that encode unobserved vertex
features through distance or similarity~\cite{Penrose03,Hoff02}. 
We study a random geometric graph generated from hard-thresholding inner products between high-dimensional latent points.
Specifically, consider i.i.d.\ Gaussian latent
vectors $x_1,\ldots,x_n\sim N(0,\Sigma)$ 
for a positive definite $\Sigma \in \R^{d \times d}$. We observe the graph with adjacency matrix $A$ given by 
$A_{ij}=\mathbbm{1}\{\ip{x_i}{x_j}\ge \zeta\}$ 
for all $i \ne j$, where the threshold $\zeta$ is chosen so that $A$ has a constant expected edge density. 
Our goal is to estimate all normalized inner products $\frac{\langle x_i, x_j \rangle}{\sqrt{\E[\langle x_i, x_j \rangle^2]}}$ where $i \ne j$. 
(The normalization is
unavoidable: simultaneously rescaling the latent vectors and the threshold
does not change the graph.)

In recent years, much of the work on random geometric graphs has focused on statistical
inference for high-dimensional models.  Two central questions are detection (or testing), which asks whether
an observed graph contains latent geometry at all, and recovery (or estimation), which asks whether that geometry can be reconstructed.  For the hard-threshold random geometric graph with spherical latent points, detection has been studied across dense and
sparse regimes~\cite{BDER16,BBN20,LMSY24,DMS26}; smooth kernel
functions have also been considered~\cite{LiuRacz23,MWX26}.  Of particular
relevance here, \cite{EM20} and \cite{BBH24} study detection for anisotropic Gaussian geometry and show
how the critical dimension for detection depends on the spectrum of the
covariance. 

Turning specifically to recovery, spectral methods have been proposed to estimate latent
distances or inner products in spherical and related latent space
models with general kernel functions~\cite{Araya19,Araya20,EMP22,MWX26}. 
In the high-dimensional hard-threshold setting,
\cite{LS23} studies a Gaussian block model which includes the isotropic Gaussian model as a special case. 
This result, together with the information-theoretic lower bound in~\cite{MZ24}, establishes the near-optimal condition $d \ll n$ for latent inner-product recovery up to polylogarithmic factors in the isotropic setting. 
In this work, we provide an approach to recover the latent inner products when the latent points are anisotropic Gaussian.

\paragraph{Spectral estimators.}
Similar to all the aforementioned positive results on latent inner-product recovery, we also study spectral methods in this work. 
To see why spectral methods constitute a natural approach, we can expand the observation $A_{ij}=\mathbbm{1}\{\ip{x_i}{x_j}\ge \zeta\}$ in a suitable basis with respect to the inner products $\ip{x_i}{x_j}$, such as the Hermite or Gegenbauer polynomial basis (depending on whether the latents are spherical or Gaussian). 
As the first-order term in such an expansion is a scalar multiple of $\ip{x_i}{x_j}$, after removing the constant part, the adjacency matrix $A$ contains a 
scalar multiple of the latent Gram matrix $(\ip{x_i}{x_j})_{i,j=1}^n$ as its linear term, and in some special cases (such as spherical and isotropic Gaussian data), the higher-order nonlinear terms can be shown to be vanishing through techniques that are also employed in the study of empirical kernel matrix approximation in various high-dimensional regimes~\cite{EK10,GMMM21,MMM22}.
This is the basis of all the above results for spectral methods. 

The setting of anisotropic Gaussian latent vectors presents new challenges.
In particular, the remaining nonlinear terms in a natural Hermite expansion of $A$ contain an
important vertex-wise effect. 
To see this, fix \(x_i \in \R^d\) and let \(x_j \sim N(0,\Sigma)\). Then the conditional connection probability $\Pp\{\langle x_i, x_j \rangle \ge \zeta \mid x_i\}$ depends
on \(x_i^\top\Sigma x_i\).  
Vertices with different values of this quadratic form
therefore have different conditional expected degrees.  The resulting
row-wise and column-wise fluctuations in $A$ can be as large in operator norm as the Gram matrix term.
In fact, this issue is already present for isotropic Gaussian vectors with $\Sigma = I_d$, because their random norms $x_i^\top\Sigma x_i = \|x_i\|^2$ affect the vertex degrees; the same effect is further amplified by anisotropy.

A simple remedy is to correct for degrees before applying spectral
truncation. Let us define $H:=I_n-\frac1n\one\one^\top$, where \(\one\) denotes the all-ones vector.  The \emph{doubly centered} matrix \(HAH\)
subtracts the row and column averages of $A$ (see \eqref{eq:double-center-degree-correct}). 
Equivalently, because \(H\one=0\), this operation removes any additive term of
the form \(a\one^\top+\one a^\top\) from $A$. 

Both the prior work~\cite{LS23} and the recent independent work \cite{FZ26} (see further discussion below) account for the degree effect in the
isotropic Gaussian model, although they implement the correction differently. 
In~\cite{LS23}, the estimator discards the leading
eigenpair of the adjacency matrix $A$ and reconstructs the latent inner products
from the next \(d\) eigenpairs, thereby removing the dominant degree direction
adaptively. Their method can therefore be viewed as a spectral form of degree
correction, even though it does not literally subtract the degrees. 
On the other hand, \cite{FZ26} motivates the degree correction via the
first-order Hoeffding projection of the Gaussian threshold kernel. Since
double-centering is the empirical analogue of removing the first-order components, they arrive at the same operation $HAH$ as what we use.

\paragraph{Main contributions.}
To recover the normalized inner products $\Big( \frac{\langle x_i, x_j \rangle}{\sqrt{\E[\langle x_i, x_j \rangle^2]}} \Big)_{i<j}$, the estimator $\hat S_d$ we use is a rank-$d$ spectral approximation of the doubly centered adjacency matrix $HAH$, and is formally defined in Equation~\eqref{eq:our-estimator}. 
Our main result, Theorem~\ref{thm:gram-recovery}, gives an explicit rate of estimation in the mean squared error that depends on the spectrum of the covariance matrix $\Sigma$. 
In particular, when the stable rank $\tr(\Sigma^2)/\opnorm{\Sigma}^2$ of
\(\Sigma\) is comparable to \(d\), the estimator $\hat S_d$ achieves the rate $O\left( \frac{d\log n}{n} + \frac 1d \right)$, which is the same as that in the independent work~\cite{FZ26} and improves upon the rate in~\cite{LS23} by a polylogarithmic factor. 

In the high-dimensional regime $d \to \infty$, the strong recovery condition $d \log n \ll n$ almost matches the impossibility condition $d \gtrsim n$ proved via rate-distortion theory in~\cite{MZ24}. 
Therefore, over any covariance class containing the isotropic model $\Sigma = I_d$, the rate we obtain is near-optimal in a minimax sense (but is not instance-wise optimal). 
More importantly, our rate of estimation allows ill-conditioned covariance matrices $\Sigma$: for example, even if $\Sigma$ has eigenvalues $i/d$ for $i = 1, \dots, d$ and thus a condition number equal to $d \to \infty$, its stable rank $\tr(\Sigma^2)/\opnorm{\Sigma}^2$ is approximately $d/3$, and so the same rate of estimation holds.

To obtain the recovery guarantee, the proof first controls the doubly centered adjacency matrix in operator norm, before passing to Frobenius recovery.  The main
technical challenge in the proof is the discontinuity of the threshold kernel.  Inner product kernels admit asymptotic spectral descriptions~\cite{ChengSinger13,FanMontanari19}, and nonasymptotic concentration is known
for Lipschitz kernels~\cite{AminiRazaee21}. 
These kernels, when the univariate function is analytic, can also be approximated by carefully chosen low-degree approximations in various high-dimensional regimes (such as the linear regime~\cite{EK10} and the polynomial scaling regime~\cite{GMMM21,MMM22}), most commonly via the trace method.
However, these results do not directly control the error terms here.  We instead build on and adapt a recent decoupling approach of~\cite{KRM25}.
In particular, we separate the first three Hermite components of the doubly centered matrix from the rest, where the linear term is the signal, the quadratic and cubic terms are controlled individually, and the complete higher-order residual is bounded as a single kernel
matrix, all using tools adapted from~\cite{KRM25}.

\paragraph{Independent work.}
As we were finishing the current work, we became aware of a recent independent work~\cite{FZ26}, which treats sparse hard-threshold random geometric graphs with isotropic spherical or Gaussian points and sharpens the rates and conditions in \cite{LS23}.
In contrast, we focus on random geometric graphs with \emph{anisotropic} Gaussian latent points, but assume a constant edge density. Our results allow a general covariance matrix $\Sigma$ whose spectrum manifests in the recovery guarantee through its stable rank.

We note that although our paper and the independent work \cite{FZ26} both use double-centering and decoupling, our treatment of the nonlinear noise terms differs significantly from that of~\cite{FZ26}. They control the entire noise term as a whole, through population integral operator estimates based on a radial spherical harmonic decomposition. We instead work with the Hermite expansion, treating the quadratic, cubic, and higher-order components separately. 

\paragraph{Organization.}
This paper is organized as follows.
Section~\ref{sec:setting} introduces the anisotropic Gaussian random geometric graph model that we study, and formalizes the desired metrics for inner-product recovery.
In Section~\ref{sec:main-results} we introduce our estimator and recovery guarantee, and provide a brief proof sketch.
Section~\ref{sec:numerical-experiments} presents numerical simulations to corroborate our theoretical findings. 
Section~\ref{sec:main-proofs} proves the theorem through intermediate operator norm
and centered Frobenius guarantees, and
Section~\ref{sec:technical-ingredients} supplies the auxiliary kernel-matrix
and concentration estimates.

\paragraph{Notation.} We use capital letters (both English and Greek) to denote matrices and lowercase letters to denote both scalars and vectors (which can either be distinguished from context, or we explicitly point out). Let $I_n$ denote the $n \times n$ identity matrix. We use $\one$ to denote the all-ones vector and use $\mathbbm{1}$ to denote the indicator function. 
Throughout, \(\opnorm{M}\) and \(\fnorm{M}\) denote the operator and Frobenius
norms of a matrix \(M\), respectively, and an unsubscripted norm
\(\norm{v}\) denotes the Euclidean norm of a vector. When used, other vector norms are explicitly specified. For a random variable \(X\) and \(p\ge1\), we write
\(\|X\|_{L^p}:=(\E|X|^p)^{1/p}\). For conditional moments, we use, for example,
\(\|X\|_{L^p(x\mid z)}:=(\E_x[|X|^p\mid z])^{1/p}\).
We define $N(0,1)$ to be the standard Gaussian distribution, denote by $\varphi(t) := \frac{1}{\sqrt{2\pi}}e^{-t^2/2}$ its probability density function (pdf), and denote by $\Phi(t)=\Pp_{g \sim N(0,1)}\{g \le t\}$ its cumulative distribution function (cdf). We also write $\bar \Phi(t) = 1-\Phi(t)$. We use the notation $C$ to denote universal constants that change from line to line, and $C_{\odot}$ to denote constants that may depend on auxiliary problem parameters that change from line to line. None of our constants exhibit any dependence on the number of samples $n$ or the data dimension $d$. We also use big-Oh notation for more informal exposition in a few places. 
For a symmetric matrix \(B\), the operation \(\Pi_d^+(B)\) retains its largest
\(\min\{d,\#\{j:\lambda_j(B)>0\}\}\) positive eigenvalues and replaces all
remaining eigenvalues by zero. 
For a function $f$, $f'$ denotes its first derivative.
When not explicitly specified, all expectations and probabilities are taken over the randomness in the data $\{x_1,\ldots,x_n\}$.
When used, $[n]$ is shorthand for the set $\{1,\ldots,n\}$.

\section{Setting and preliminaries}
\label{sec:setting}

We begin by defining the anisotropic Gaussian random geometric graphs that we study in this work.

\subsection{Anisotropic random geometric graphs}

Let \(x_1,\ldots,x_n\in\R^d\) be i.i.d.\ samples from \(N(0,\Sigma)\),
where $\Sigma \in \R^{d \times d}$ and we assume without loss of generality that \(\Sigma\) is positive definite\footnote{This is without loss of generality because if $\Sigma$ is singular, we may restrict the latent vectors to their support which is a linear subspace and this restriction preserves all pairwise inner products.}.
Throughout, we focus on $d \le n$ since it is impossible to achieve consistent recovery if $d > n$ as discussed in the introduction.
We denote the data matrix by
\[
X=\begin{bmatrix}x_1^\top\\ \vdots\\ x_n^\top\end{bmatrix} \in \R^{n \times d},
\]
and define $$\tau_k := \tr(\Sigma^k)$$ for all integers $k \geq 1$. (Observe that in the special case of isotropic covariance, we have $\tau_k = d$ for all $k$.)
We observe a random geometric graph consisting of $n$ vertices, where vertex $i$ corresponds to sample $x_i$.
The edges are specified by the following definition of the adjacency matrix:
\[
A_{ij}=\mathbbm{1}\{\ip{x_i}{x_j}\ge \zeta\} \cdot \mathbbm{1}\{i \neq j\},
\]
where $\zeta > 0$ is a parameter that we will specify shortly.
This is a \emph{threshold-based random geometric graph}: an edge $(i,j)$ is present if the inner product between corresponding samples $x_i$ and $x_j$ is sufficiently large.
We write
\[
p:=\Pp\{\ip{x_i}{x_j}\ge\zeta\}
\]
for the expected edge density of $A$. 
Throughout this work we set
$\zeta$ in order to ensure a \emph{constant edge density}, i.e., $p\asymp1$.
It can be verified that this requirement corresponds to setting $\zeta = O( \sqrt{\tau_2} )$. %

\subsection{The main goal: Recovery of inner products}

In this section, we state the formal goal of this work, which is \emph{inner-product recovery}. 
In view of the scale ambiguity of the model\footnote{Observe that for any $s > 0$, replacing the tuple $(X,\Sigma,\zeta)$ by $(\sqrt{s} X, s \Sigma, s \zeta)$ does not change the adjacency matrix.}, our goal is to recover the scale-invariant normalized inner products.  More precisely, since
\begin{equation*}
\tau_2=\tr(\Sigma^2)= \E[\langle x_i, x_j \rangle^2]
\end{equation*}
for any $i < j$, we aim to recover
$\left(\frac{\ip{x_i}{x_j}}{\sqrt{\tau_2}}\right)_{1 \leq i < j \leq n}$.  
This specific rescaling is chosen in order to make the typical value of the inner products to be at constant order.
This goal can be equivalently expressed in terms of a matrix estimation problem: we wish to estimate the normalized data Gram matrix $XX^T/\sqrt{\tau_2}$. 
We provide a formal definition below.
\begin{definition}\label{def:expected-ip-recovery}
Denoting our candidate estimator by $\hat S_d$, we will measure estimation error of the inner products by the normalized squared Frobenius norm as defined below:
\[
\frac{1}{n^2} \fnorm{\hat S_d - \frac{XX^T}{\sqrt{\tau_2}}}^2 = \frac{1}{n^2} \sum_{i=1}^n \sum_{j=1}^n \left((\hat S_d)_{ij} - \frac{\ip{x_i}{x_j}}{\sqrt{\tau_2}}\right)^2.
\]
Accordingly, we state that an estimator achieves \emph{strong expected inner-product recovery} with respect to a sequence $(n,d=d_n, \Sigma=\Sigma_n)_{n \geq 1}$ when the mean squared error satisfies
\[
\E\left[\frac{1}{n^2} \fnorm{\hat S_d - \frac{XX^T}{\sqrt{\tau_2}}}^2\right] \to 0,
\]
where convergence is defined along the specified asymptotic.
\end{definition}
Expanding the definition of the Frobenius norm as above makes clear why it is a natural estimation error metric: leaving aside the contribution of the diagonal error terms, the normalized squared Frobenius norm considers the average of the squared inner-product recovery error over all edges.
Note that expected inner-product recovery can be converted into inner-product recovery in probability, i.e., convergence in probability to $0$ of the error $\frac{1}{n^2} \fnorm{\hat S_d - \frac{XX^T}{\sqrt{\tau_2}}}^2$, via a simple application of Markov's inequality.
(This application, however, yields slow logarithmic rates of convergence of the probability to $1$, as in~\cite{KRM25}.) 

In what follows, we describe our estimator and recovery guarantees, noting that strong expected inner-product recovery follows directly from the provided non-asymptotic error rates that hold for finite values of $n, d$ and $\Sigma$.

\section{Main results}\label{sec:main-results}

We build on the simple spectral estimator that was proposed by~\cite{LS23} for the case of random geometric graphs with isotropic latent points. 
The spectral estimator is essentially the rank-$d$ spectral approximation of the centered adjacency matrix, and~\cite{LS23} show that the error of the spectral estimator is sufficiently small provided that a suitable \emph{kernel matrix approximation} result holds on the adjacency matrix $A$.
To see the connection, recall that the off-diagonal entries of $A$ are given by the threshold function 
$$
A_{ij} = \mathbbm{1}\left(\frac{\ip{x_i}{x_j}}{\sqrt{\tau_2}} \ge \frac{\zeta}{\sqrt{\tau_2}}\right) = \mathbbm{1}\left(Z_{ij} \ge t\right),
$$ 
where we define $Z_{ij} := \ip{x_i}{x_j}/\sqrt{\tau_2}$ and $t := \zeta/\sqrt{\tau_2}$. 
Since $\E Z_{ij}^2=1$, the central limit theorem shows that $Z_{ij}$ is asymptotically standard Gaussian as $d \to \infty$ (if $\Sigma$ is not ill-conditioned). 
Letting $p_{\mathrm G}:=\bar\Phi(t) = 1-\Phi(t)$, the degree-one Hermite polynomial approximation of $\mathbbm{1}\left(Z_{ij} \ge t\right)$ is given by $p_{\mathrm G} + \varphi(t) Z_{ij}$. 
As a result, the entrywise first-order (affine) approximation of $A$ is then given by 
\begin{equation}
A \approx p_{\mathrm G} \mathbf{1}\mathbf{1}^\top + \varphi(t) \frac{XX^T}{\sqrt{\tau_2}}. \label{eq:A-linear-approx-hermite}
\end{equation}
Therefore, to show that the inner products can be recovered from $A$ via a spectral method, it suffices to show that the above approximation error is sufficiently small---the normalized operator norm of the higher-order (degree-two and above in $Z_{ij}$) component of $A$ decays to $0$ at a sufficiently fast rate as a function of $n$. 

Interestingly, 
the decomposition \eqref{eq:A-linear-approx-hermite} is not sufficient for understanding an anisotropic random geometric graph, even in the dense regime with Gaussian latent points that we study. 
We would ostensibly aim to show that the higher-order component of $A$ with respect to the Hermite polynomial expansion decays.
This turns out not to be the case; the quadratic component is comparable in operator norm to the linear component---see Section~\ref{sec:quadratic-hermite} for an explicit calculation.
This suboptimality echoes similar suboptimalities in approximating empirical kernel matrices in the polynomial scaling regime $n \propto d^q$ through the univariate Hermite decomposition; see~\cite{KRM25}.

\subsection{Degree-corrected, double-centered estimator}

We consider a simple \emph{degree-corrected}, or \emph{double-centered}, matrix-valued estimator.
Recall that we defined $H=I_n-\frac1n\one\one^\top$.
To understand why we are referring to this operation as ``degree-correcting" or ``vertex-centering", observe that
\begin{align}
(HAH)_{ij} &= A_{ij} - \frac{1}{n} \sum_{i' = 1}^n A_{i'j} - \frac{1}{n} \sum_{j' = 1}^n A_{ij'} + \frac{1}{n^2} \sum_{i' = 1}^n \sum_{j'= 1}^n A_{i'j'} \notag \\
&= A_{ij} - \frac{\deg(i)}{n} - \frac{\deg(j)}{n} + \frac{2}{n^2} |E(A)|, \label{eq:double-center-degree-correct}   
\end{align}
where $E(A)$ denotes the edge set of the graph $A$. 
We will later show that the quadratic component of $HAH$ in its Hermite expansion is significantly smaller as compared to that of the original adjacency matrix $A$.

Since $H \mathbf{1} = 0$ and in view of \eqref{eq:A-linear-approx-hermite}, we note that the matrix $HAH$ no longer has a degree-zero term. 
Therefore, our estimator takes the following form in the idealized scenario where the normalized threshold $t$, and hence the Gaussian reference probability $p_{\mathrm G}$, is known:
\begin{align}\label{eq:our-idealized-estimator}
\tilde S_d = \Pi_d^+\left(\frac{1}{\varphi(t)} H(A + p_{\mathrm G} I_n) H\right),
\end{align}
where the $+p_{\mathrm G} I_n$ term fills in the diagonal of $A$ as we aim to estimate $\frac{XX^T}{\sqrt{\tau_2}}$, and $\Pi_d^+$ denotes the operation of taking the rank-$d$ spectral approximation. 
More precisely, for a symmetric matrix \(B\) with eigendecomposition
\(B=\sum_{k=1}^n\lambda_k u_k u_k^\top\) where
\(\lambda_1\ge\cdots\ge\lambda_n\), define
\[
\Pi_d^+(B):=\sum_{k=1}^{d}(\lambda_k)_+u_k u_k^\top,
\qquad \text{where } (x)_+:=\max\{x,0\}.
\]

In reality, neither $t$ nor $p$ is known.  We use the plug-in quantities
\begin{align*}
\hat p &:= \min\{1 - n^{-2},\max\{n^{-2},\hat p_0\}\} \text{ where } \hat p_0 := \frac{2|E|}{n(n-1)}, \notag \\
\hat t &:= \bar\Phi^{-1}(\hat p), \\
\hat\beta_1 &:= \varphi (\hat t). \notag
\end{align*}
Here $\hat p$ estimates the true edge probability $p$ directly and is also
used as a plug-in estimate of $p_{\mathrm G}$\footnote{The clipping operations for the estimator of the edge probability ensure that the estimator is well-defined on the exceptional empty or complete-graph events. These events will occur with vanishing probability in our random graph setup.}; correspondingly, $\hat t$ and
$\hat\beta_1$ estimate $t$ and the linear Hermite coefficient
$\beta_1=\varphi(t)$.
We plug these estimates in to obtain the ultimate estimator
\begin{align}\label{eq:our-estimator}
\hat S_d
:=\Pi_d^+\left(\frac{1}{\hat\beta_1}H(A+\hat p I_n)H\right),
\end{align}

\subsection{Inner-product recovery guarantee}

We now present our inner-product recovery guarantee in terms of a non-asymptotic bound on the normalized squared Frobenius error.
The error bound is expressed in terms of
the \emph{stable rank} of the covariance matrix $\Sigma$ defined as
\begin{equation*}
\stablerank:=\frac{\tr(\Sigma^2)}{\opnorm{\Sigma}^2}
=\frac{\tau_2}{\opnorm{\Sigma}^2}.
\end{equation*}
The stable rank satisfies $1\le\stablerank\le d$.
Our main theorem bounds the Frobenius recovery error in terms of
$n$, $d$, and $\stablerank$.

\begin{theorem}
\label{thm:gram-recovery}
Consider any sequence $\{(n,d,\Sigma)\}_{n \geq 1}$ with \(n \to\infty\), and assume that \(t = \zeta/\sqrt{\tau_2}\) stays in a fixed compact interval.
There are constants $C,c>0$ depending only on this interval such that the following holds. 
Assume that $(d \log n)/n + d/\stablerank^2 \leq c$. For all sufficiently large \(n\), 
\begin{equation}
\E\left[
\frac1{n^2}\fnorm{\hat S_d-\frac{XX^\top}{\sqrt{\tau_2}}}^2
\right]
\le C\left(
\frac{d\log n}{n}
+\frac{d}{\stablerank^2}
\right).
\label{eq:normalized-uncentered-recovery}
\end{equation}
Observe that this ensures strong expected inner-product recovery (according to Definition~\ref{def:expected-ip-recovery}) provided that the right-hand side of \eqref{eq:normalized-uncentered-recovery} tends to zero. 
\end{theorem}

The two terms in the rate in \eqref{eq:normalized-uncentered-recovery} arise from different sources. The first term $\frac{d\log n}{n}$ is a finite-sample
fluctuation term which vanishes as the sample size $n$ grows. The second term $\frac{d}{\stablerank^2}$ controls the nonlinear part of the threshold kernel: the quadratic, cubic,
and higher-order Hermite components are all bounded at the common scale
$d/\stablerank^2$. 
In particular, if 
% To compare with prior work~\cite{LS23}, we specialize
% Theorem~\ref{thm:gram-recovery} to sequences satisfying
\(\stablerank\asymp d\), 
the rate becomes
\[
O \left( \frac{d\log n}{n}+\frac1d \right) ,
\]
and hence the Frobenius mean-squared error vanishes if $d\to\infty$ and $n \gg d \log n$.
In the dense regime, this rate improves the result in~\cite{LS23} by polylogarithmic factors and matches that in the independent work~\cite{FZ26}. 
The assumption \(\stablerank\asymp d\) includes uniformly well-conditioned covariance matrices but
is strictly weaker: \(\lambda_{\min}(\Sigma)\) may tend to zero and the
condition number may diverge.

Finally, we note that while the strong recovery condition $n \gg d \log n$ matches the impossibility condition $n \lesssim d$ in \cite{MZ24} up to a logarithmic factor, the \emph{instance-optimal} condition for recovery remains unknown for a specific sequence $\Sigma = \Sigma_n$. This interesting question is beyond the scope of the current work, and we leave it to future research.

\subsection{Proof sketch}\label{sec:proof-sketch}

In this section, we outline the main steps in the proof of Theorem~\ref{thm:gram-recovery}.
The full proof is provided in Section~\ref{sec:main-proofs}.

\paragraph{Reduction to operator norm error for the idealized double-centered estimator.} Because kernel approximation-theoretic bounds are often proved in operator norm~\cite{EK10,GMMM21,MMM22}, it will be convenient to first work with an operator norm analysis and then derive a suitable bound on the Frobenius error.
Moreover, it will be convenient to reduce to an error analysis of the idealized estimator defined in Equation~\eqref{eq:our-idealized-estimator} that assumes oracle access to the probability $p_G$ or, equivalently, the threshold $t$.
We describe the steps of this reduction below.
We refer to the estimators before spectral truncation is applied as $\hat S$ and $\tilde S$ respectively.
We also define the \emph{double-centered estimand} and denote the original estimand respectively as follows:
\[
S := \frac{H XX^T H}{\sqrt{\tau_2}}, \qquad G := \frac{XX^T}{\sqrt{\tau_2}}.
\]
\begin{enumerate}
\item First, we expect the original estimand and its double-centered version to be very close to each other. Lemma~\ref{lem:centered-uncentered} shows this formally by bounding $\fnorm{S - G}$; therefore, it suffices to bound the Frobenius error $\frac{1}{n} \fnorm{\hat S_d - S}$.
\item Second, since the rank of the double-centered estimand $S$ is at most $d$ and we defined $\hat S_d = \Pi_d^+(\hat S)$, we have $\fnorm{\hat S_d - S} \lesssim \sqrt{d} \opnorm{\hat S - S}$, which is a simple linear algebra fact stated in Lemma~\ref{lem:op-to-frobenius-truncation}. Therefore, it suffices to bound the normalized operator norm error of the original untruncated estimator, i.e.,  $\frac{\sqrt{d}}{n} \opnorm{\hat S - S}$. This is done in Proposition~\ref{prop:operator-denoising}.
\item Finally, the proof of Proposition~\ref{prop:operator-denoising} decomposes the error into the idealized estimation error $\opnorm{\tilde S - S}$ and the error arising from estimating edge density $\opnorm{\tilde S - \hat S}$.
\end{enumerate}
Therefore, the heart of our analysis lies in characterizing the idealized estimation error in operator norm, i.e., controlling $\opnorm{\tilde S - S}$.
We next outline how this is done.

\paragraph{The Hermite decomposition and approximation analysis.}

If \(\stablerank\to\infty\), i.e., \(\|\Sigma\|_{\op}/\sqrt{\tau_2}\to0\), the central limit theorem implies that the random variable corresponding to the normalized inner product, i.e.,
\[
Z_{ij}:=\frac{\ip{x_i}{x_j}}{\sqrt{\tau_2}},
\]
is asymptotically standard Gaussian for each fixed pair \(i\ne j\).  This limit
motivates the following Hermite polynomial expansion of the centered threshold function in $L^2(N(0,1))$:
\begin{equation}
\mathbbm{1}\{z\ge t\}-p_{\mathrm G}
=\sum_{k\ge1}\beta_k\He_k(z),\text{ where }
\beta_k=\frac{\varphi(t)\He_{k-1}(t)}{k!} ,
\label{eq:hermite-coefficients}
\end{equation}
and \(\{\He_k\}_{k \geq 1}\) denote the probabilists' Hermite polynomials.  Indeed, if
\(g\sim N(0,1)\), then for every index \(k\ge1\),  we have
\[
\frac1{k!}\E\!\left[
(\mathbbm{1}\{g\ge t\}-p_{\mathrm G})\He_k(g)
\right]
=
\frac1{k!}\int_t^\infty \He_k(z)\varphi(z)\,dz
=
\frac{\varphi(t)\He_{k-1}(t)}{k!},
\]
where the last equality follows from the identity
\((\varphi\He_{k-1})'=-\varphi\He_k\). 
(Observe that the coefficients $\{\beta_k\}_{k \geq 1}$ are dimensionless when $t$ stays in a fixed compact interval, as we have assumed in this work. Moreover, in this case $\beta_1$ is bounded above and away from zero and does not scale with $d$; therefore, we can divide relevant quantities by $\beta_1$.)

We emphasize that \eqref{eq:hermite-coefficients} is an orthogonal expansion with respect to the standard Gaussian distribution, not the actual law of $Z_{ij}$, and the convergence is only justified in $L^2(N(0,1))$.
Accordingly, in the proof we never use \eqref{eq:hermite-coefficients} as an infinite
series evaluated at $Z_{ij}$; the infinite Hermite expansion is invoked only after conditioning on one vector $x_i$ so that $Z_{ij}$ is exactly Gaussian.

Equation~\eqref{eq:hermite-coefficients} together with the definition of the idealized estimator $\tilde S$ (see also Equation~\eqref{eq:our-idealized-estimator}) implies that, to control the desired error $\opnorm{\tilde S - S}$, it suffices to bound the operator norm of the \emph{double-centered version of} the off-diagonal component of the nonlinear (higher-order) kernel matrix whose $(i,j)$-th entry is $\sum_{k \geq 2} \beta_k \He_k(Z_{ij})$; the matrix is denoted as $\sum_{k \geq 2} \beta_k H \Delta^{(k)} H$.
This is the technical crux of our proof, and has to be done carefully in a series of steps that we outline below:
\begin{enumerate}
\item Unlike the related literature on empirical kernel matrix approximation in the polynomial scaling regime~\cite{GMMM21,MMM22,KRM25}, it is insufficient to control the individual higher-order Hermite kernel matrices and apply the triangle inequality---this leads to a series on $k$ that is not summable.
Instead, the ``whole-tail residual" needs to be controlled directly.
\item The quadratic and cubic terms, corresponding to $H \Delta^{(2)} H$ and $H \Delta^{(3)} H$, turn out to need to be characterized separately from the rest of the whole-tail residual. For the cubic term, Proposition~\ref{prop:krm-fixed-degree} combines the decoupling argument of~\cite{KRM25} with an exact 
representation of its correlation matrix using tensor algebra.
On the other hand, the quadratic term is more delicate; as informally discussed while motivating our estimator in Section~\ref{sec:main-results}, the double-centering operation is crucial in order to remove the prohibitively large contribution from the entry-wise conditional expectation of the quadratic Hermite kernel over one of the arguments.
We show in Proposition~\ref{prop:quadratic-centered} that, in fact, $\opnorm{H \Delta^{(2)} H} \ll \opnorm{\Delta^{(2)}}$ as desired.
The proof of Proposition~\ref{prop:quadratic-centered} also invokes~\cite{KRM25}, but only after replacing the quadratic kernel by its canonical version, obtained by subtracting the first-order contributions associated with its two arguments. This replacement is used only in the analysis and is not part of our estimator. 
\item After handling the quadratic and cubic terms as outlined above, the ``whole tail residual matrix" $\sum_{k \geq 4} \beta_k H \Delta^{(k)} H$ is characterized by Proposition~\ref{prop:tail}.
The proof of Proposition~\ref{prop:tail} also invokes~\cite{KRM25} together with a careful characterization of the correlation matrix entries corresponding to the whole-tail kernel matrix, and in particular several properties of the Hermite coefficients $\{\beta_k\}_{k \geq 1}$ that we prove in the lemmas accompanying the proof of Proposition~\ref{prop:tail}.
\end{enumerate}
Ultimately, putting together Proposition~\ref{prop:quadratic-centered}, Proposition~\ref{prop:krm-fixed-degree} and Proposition~\ref{prop:tail} yields the desired bound on the idealized estimation error in operator norm, i.e., $\opnorm{\tilde S - S}$.
Combining this with the outlined path to Frobenius recovery yields Theorem~\ref{thm:gram-recovery}.

\section{Numerical experiments}
\label{sec:numerical-experiments}

In this section, we present numerical simulations to support the above theoretical results and to suggest interesting research directions. 
For each covariance matrix $\Sigma$, we draw i.i.d.\ 
$x_1,\ldots,x_n\sim N(0,\Sigma)$, and set
$$
A_{ij}=\mathbbm{1}\{\ip{x_i}{x_j}\ge\sqrt{\tau_2}\}
$$ 
for $i\ne j$ and $A_{ii}=0$. Thus the normalized threshold is fixed at
$t = \zeta/\sqrt{\tau_2}=1$, and the distribution of the graph is invariant under rescaling of $\Sigma$. 
Moreover, we may take $\Sigma$ to be diagonal without loss of generality since an orthogonal transformation 
preserves the inner products. 
Given a graph $A$, we compute the estimator
$\hat S_d$ in \eqref{eq:our-estimator} and measure the error by the normalized squared Frobenius norm in Definition~\ref{def:expected-ip-recovery}. 

Since the result in Theorem~\ref{thm:gram-recovery} depends on both $d$ and $n$, we consider two joint scalings of them for the experiments in Figure~\ref{fig:simulation-sweeps}. We take geometrically spaced $d\in\{64,91,128,181,256\}$ for all the experiments. For the three plots in the top row of Figure~\ref{fig:simulation-sweeps}, we let $n=\lceil d^{3/2}\rceil\in\{512,869,1449,2436,4096\}$, while for those in the bottom row of Figure~\ref{fig:simulation-sweeps}, we let $n=\lceil d^2/10\rceil\in\{410,829,1639,3277,6554\}$. 
Each point in Figure~\ref{fig:simulation-sweeps} is an average over $10$ independent repetitions. 

\begin{figure}[!htb]
\centering
\includegraphics[width=\textwidth]{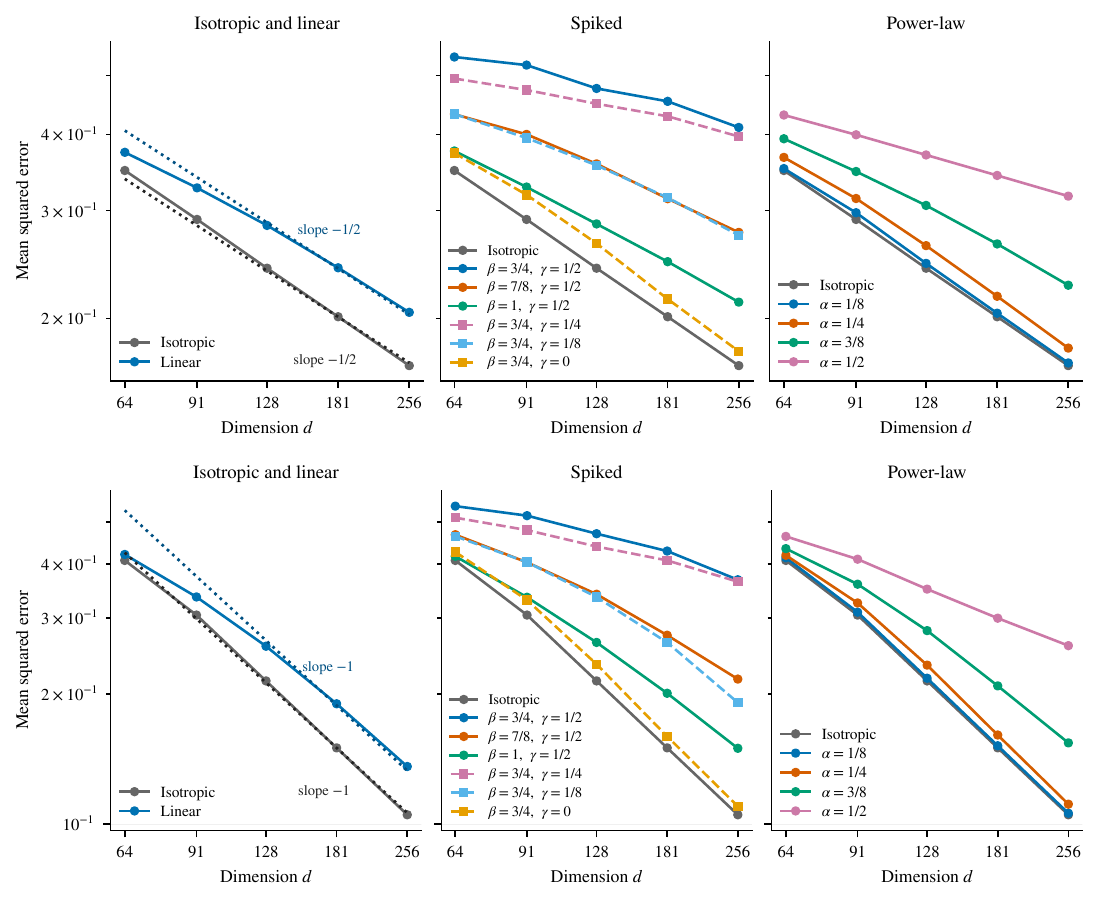}
\caption{Normalized squared Frobenius error of $\hat S_d$ in the logarithmic scale. Top row: $n=\lceil d^{3/2}\rceil$.
Bottom row: $n=\lceil d^2/10\rceil$. 
% The dotted lines in the first column have slopes $-1/2$ (top)
% and $-1$ (bottom), and are fitted to their respective error curves.
}
\label{fig:simulation-sweeps}
\label{fig:simulation-joint-growth}
\label{fig:simulation-quadratic-growth}
\end{figure}
% \FloatBarrier

\paragraph{Isotropic and linear covariance.}
In the left column of Figure~\ref{fig:simulation-sweeps}, we consider the isotropic case $\Sigma=I_d$ and the covariance with linearly decaying eigenvalues $\Sigma=\diag(1,\frac{d-1}{d},\frac{d-2}{d},\ldots,\frac{1}{d})$.
The latter $\Sigma$ has condition number $d$, but stable rank $\stablerank = \frac{(d+1)(2d+1)}{6d} \approx d/3$. Therefore, both models satisfy the $O(\frac{d\log n}{n}+\frac1d)$ guarantee by Theorem~\ref{thm:gram-recovery} and the discussion thereafter. With $n=\lceil d^{3/2}\rceil$ and $n=\lceil d^2/10\rceil$, the error bounds become $\tilde O(d^{-1/2})$ and $\tilde O(d^{-1})$ respectively. 
Indeed, these are approximately the rates of the error decay as the dimensions grow in the experiments, confirmed by the linear fits (dotted lines) with slopes $-1/2$ (top) and $-1$ (bottom) in the respective plots.

\paragraph{Spiked covariance.}
In the middle column of Figure~\ref{fig:simulation-sweeps}, we take the covariance with 
$k=\lfloor d^\beta/2\rfloor$ eigenvalues equal to $\lambda=2d^\gamma$
and the remaining $d-k$ eigenvalues equal to $1$, where 
\[
\begin{aligned}
\gamma=1/2,\qquad &\beta\in\{3/4,7/8,1\} \qquad (\text{solid lines}),\\
\beta=3/4,\qquad &\gamma\in\{1/4,1/8,0\} \qquad (\text{dashed lines}).
\end{aligned}
\]
The first group varies the number of spikes at fixed strength, while 
the second varies the strength of a fixed number of spikes. 
We then have
$$
\tau_2=d+k(4d^{2\gamma}-1),\qquad
\stablerank=k+\frac{d-k}{4d^{2\gamma}}.
$$
Theorem~\ref{thm:gram-recovery} provides the error bound
\[
O\left(\frac{d\log n}{n}+d^{1-2\beta}\right)
\]
for $\gamma=1/2$ and $\beta\in\{3/4,7/8,1\}$, and
\[
O\left(\frac{d\log n}{n}+d^{\min\{4\gamma-1,-1/2\}}\right)
\]
for $\beta=3/4$ and $\gamma\in\{1/4,1/8,0\}$.
Therefore, all six cases have error bounds between $\tilde O(d^{-1/2})$ and $\tilde O(d^{-1})$ along the sequences $n=\lceil d^{3/2}\rceil$ and $n=\lceil d^2/10\rceil$. 
While it is difficult to see the precise rates of error decay in the plots, the experiments demonstrate smaller estimation errors for larger $\beta$ (more spikes) and smaller $\gamma$ (weaker spikes), corresponding to a more isotropic covariance with a larger stable rank in our theory.

We remark that the cases $\gamma=1/4$ and $\gamma=1/8$ have the same asymptotic upper bound by Theorem~\ref{thm:gram-recovery}, but their finite-dimensional errors differ in the experiments. The sharp rates of estimation for such spiky covariances remain inconclusive and are interesting questions for future research.

\paragraph{Power-law covariance.}
In the right column of Figure~\ref{fig:simulation-sweeps}, we consider the covariance with eigenvalues following a power law $\Sigma_{ii}=i^{-\alpha}$ for $i\in[d]$, where
$\alpha\in\{1/8,1/4,3/8,1/2\}$. For $\alpha < 1/2$, we have 
$\stablerank=\tau_2\asymp d^{1-2\alpha}$, so
Theorem~\ref{thm:gram-recovery} gives the error bound 
\[
O\left(\frac{d\log n}{n}+d^{4\alpha-1}\right),
\]
which guarantees consistency if $\alpha < 1/4$. 
% For $n=\lceil d^{3/2}\rceil$ and $n=\lceil d^2/10\rceil$, this becomes
% $O(d^{-1/2}\log d+d^{4\alpha-1})$ and
% $O(d^{-1}\log d+d^{4\alpha-1})$, respectively.
% At the endpoint $\alpha=1/2$, we have $\tau_2\asymp\log d$, and the bound is
% instead $O(d\log n/n+d/(\log d)^2)$.
% These bounds guarantee consistency for $\alpha<1/4$.
On the other hand, the observed error decay for $\alpha\ge1/4$ is not explained by our theory. 
Hence, the stable rank may not be the optimal quantity to characterize the rate of estimation in the power-law case. The question is outside the scope of this work and is left to future research.

\section{Proof of the main result}
\label{sec:main-proofs}

We introduce shorthand notations
\begin{equation*}
\mu:=\opnorm{\Sigma}, \qquad
\effectivedim:=\frac{\tau_2^2}{\tau_4}.
\end{equation*}
Here $\effectivedim$ is the effective dimension of $\Sigma^2$ and satisfies
\begin{equation}
1\le\stablerank\le\effectivedim
\le\frac{\tau_1^2}{\tau_2}\le d,
\qquad
\effectivedim\le\stablerank^2.
\label{eq:technical-effective-dimension-ordering}
\end{equation}
Define the proof abbreviations
\begin{equation}
\begin{aligned}
\eta_n&:=\sqrt{\frac{\log n}{\effectivedim}}
+\frac{\log n}{\stablerank},
&\qquad
\mathcal B_n&:=\tau_1+\sqrt{\tau_2\log n}+\mu\log n,\\
\kappa_3&:=\frac{1}{\stablerank\effectivedim}
+\frac{1}{\stablerank^3}.
\end{aligned}
\label{eq:proof-abbreviations}
\end{equation}
As already highlighted in Section~\ref{sec:proof-sketch}, the proof of Theorem~\ref{thm:gram-recovery} reduces the desired Frobenius norm guarantee to a scaled operator norm, and then analyzes the operator norm error (which we refer to as a denoising error).
We first present the denoising argument.

\subsection{Operator-norm denoising error}

Recall that we denoted by
\[
S:=\frac{HXX^\top H}{\sqrt{\tau_2}} \text{ and } G:=\frac{XX^\top}{\sqrt{\tau_2}}
\]
the double-centered and original estimand respectively.
We note that $\opnorm{S}\asymp n/\sqrt{\stablerank}$ with high probability\footnote{Lemma~\ref{lem:gaussian-concentration-inputs} (in particular, Equation~\eqref{eq:input-gaussian-matrix-bound}) gives $\opnorm{S}\lesssim \frac{n\mu}{\sqrt{\tau_2}}$. 
Conversely, let $v$ be a top eigenvector of $\Sigma$. Then
$\opnorm{S}
\ge \frac{\|HXv\|^2}{\sqrt{\tau_2}}
\overset{d}{=}
\frac{\mu\chi^2_{n-1}}{\sqrt{\tau_2}},$
yielding the matching lower bound with high probability.}. As $\stablerank\le d$, we use the stronger normalization
$\sqrt d/n$ for the denoising bound on 
\[
\frac{\sqrt d}{n}\opnorm{\hat S-S}.
\]
We make this choice to ensure that we can successfully convert the operator norm bound to the desired Frobenius norm bound, as explained in Section~\ref{sec:proof-sketch}.
We now state and prove the main result of this subsection, which controls this denoising error.

\begin{proposition}
\label{prop:operator-denoising}
Assume that $t = \zeta/\sqrt{\tau_2}$ stays in a fixed compact interval.  There exists a sufficiently small constant $c_t>0$ 
depending only on this interval, such that, for all sufficiently
large $n$ satisfying 
\begin{equation}
d \le n \log n, \qquad \frac{d}{\stablerank^2}
\le c_t,
\label{eq:rate-bounded-by-constant}
\end{equation}
we have

\begin{equation*}
\left(
\E\left[
\left(\frac{\sqrt d}{n}\opnorm{\hat S-S}\right)^2
\right]
\right)^{1/2}
\le C_t\left(
\sqrt{\frac{d\log n}{n}}
+\frac{\sqrt d}{\stablerank}
\right).
\end{equation*}
\par
\end{proposition}

\begin{proof}[Proof of Proposition~\ref{prop:operator-denoising}]
Recall that we defined the idealized estimator as 
\[
\tilde S=\frac{1}{\beta_1}H(A+p_{\mathrm G}I_n)H,\text{ where } \beta_1=\varphi(t).
\]
By Minkowski's inequality, we have
\[
\E\left[\left(\opnorm{\hat S - S}\right)^2\right]^{1/2} \leq \underbrace{\E\left[\left(\opnorm{\tilde S - S}\right)^2\right]^{1/2}}_{\text{idealized estimation error}} + \underbrace{\E\left[\left(\opnorm{\tilde S - \hat S}\right)^2\right]^{1/2}}_{\text{edge density estimation error}}.
\]
We first control the idealized estimation error $\E\left[\left(\opnorm{\tilde S - S}\right)^2\right]^{1/2}$.

\paragraph{Hermite decomposition.}
The first step is to express the idealized estimator $\tilde S$ in terms of an entry-wise Hermite polynomial decomposition.
We define the \emph{whole-tail residual function} 
\[
q_{\ge4}(u)
:=
\mathbbm{1}\{u\ge t\}-p_{\mathrm G}-\sum_{k=1}^3\beta_k\He_k(u)
\]
and the collection of random variables \(Z_{ij} :=\ip{x_i}{x_j}/\sqrt{\tau_2}\) for each \(i\ne j\).
Rearranging the above display and recalling the definition of the adjacency matrix $A$, we note that
\[
A_{ij} = p_{\mathrm G} + \sum_{k=1}^3\beta_k\He_k(Z_{ij})
+q_{\ge4}(Z_{ij}) \text{ for all } i \neq j.
\]
We define $T_{\ge 4}$ to be the matrix with entries given by $q_{\geq 4}(Z_{ij}) \mathbbm{1}\{i \ne j\}$.
Similarly, for $k \in \{1,2,3\}$, we define $\Delta^{(k)}$ to be the matrix with entries given by $\He_k(Z_{ij})\mathbbm{1}\{i \ne j\}$.
Note that the diagonal entries of all of these matrices are zero.
We then have
\begin{align*}
\tilde S = \frac{1}{\beta_1} H(A + p_G I_n) H &= \frac{1}{\beta_1} H(p_G \one \one^\top + \sum_{k =1}^3 \beta_k \Delta^{(k)} + T_{\ge 4})H\\
&=\sum_{k=1}^3\frac{\beta_k}{\beta_1} H\Delta^{(k)}H+\frac{1}{\beta_1}HT_{\ge4}H,
\end{align*}
where we used the fact that $H \one = 0$.
Subtracting $S$ from both sides of the above display and noting that $\He_1(u) = u$, we then have
\begin{align}\label{eq:conditional-hermite-decomposition}
\tilde S - S = -HD_XH +\frac{\beta_2}{\beta_1}H\Delta^{(2)}H +\frac{\beta_3}{\beta_1}H\Delta^{(3)} H +\frac1{\beta_1}HT_{\ge4}H,
\end{align}
where we defined the diagonal matrix 
\[
D_X=\diag\left(\frac{\|x_1\|^2}{\sqrt{\tau_2}},\ldots,
\frac{\|x_n\|^2}{\sqrt{\tau_2}}\right).
\]
By the triangle inequality, it suffices to control the operator norm of the four matrices in the right hand side of Equation~\eqref{eq:conditional-hermite-decomposition}.
We do this next.

\paragraph{Controlling the diagonal and nonlinear error matrices.}
We begin by bounding the operator norm of the diagonal error term $H D_X H$.
Because $\opnorm{H} \leq 1$, we have
\[
\opnorm{HD_XH}\le \opnorm{D_X}
=
\max_i\frac{\|x_i\|^2}{\sqrt{\tau_2}}.
\]
Then, Lemma~\ref{lem:gaussian-concentration-inputs} (in particular, Equation~\eqref{eq:input-row-fourth-expectation}) directly yields
\begin{equation*}
\left(
\E\left[
\left(\frac{\sqrt d}{n}\opnorm{HD_XH}\right)^2
\right]
\right)^{1/2}
\le C\sqrt{\frac d{\tau_2}}\frac{(\tau_1 + \sqrt{\tau_2 \log n} + \mu \log n)}{n} =: C\sqrt{\frac d{\tau_2}}\frac{\mathcal B_n}{n},
\end{equation*}
where we recall the definition of the shorthand notation $\mathcal B_n$ from Equation~\eqref{eq:proof-abbreviations}.

We next bound the quadratic, cubic, and whole-residual terms.  
For the quadratic term, Proposition~\ref{prop:quadratic-centered} directly gives
\begin{equation*}
\left(\E\opnorm{H\Delta^{(2)}H}^2\right)^{1/2}
\le C\left(
\frac{n}{\stablerank}
+(1+\eta_n)\sqrt{n\log n}
+\eta_n
\right),
\end{equation*}
where we recall the definition of the shorthand notation $\eta_n$ from Equation~\eqref{eq:proof-abbreviations}. 
For the cubic term, Proposition~\ref{prop:krm-fixed-degree} upper bounds $\opnorm{\Delta^{(3)}}$.
Combining this with the fact that \(\opnorm{H}\le1\) yields
\begin{equation*}
\left(\E\opnorm{H\Delta^{(3)}H}^2\right)^{1/2}
\le C\left(
\sqrt{n(1+\eta_n^3)\log n}+n\sqrt{\kappa_3}
\right).
\end{equation*}
Finally, Proposition~\ref{prop:tail} upper bounds the operator norm of the whole tail-residual, i.e. $\opnorm{T_{\ge 4}}$.
Using, again, that $\opnorm{H} \leq 1$ yields
\begin{equation*}
\left(\E\opnorm{HT_{\ge4}H}^2\right)^{1/2}
\le C_t\left(
\sqrt{n(1+\eta_n^3)\log n}+\frac n\effectivedim
\right).
\end{equation*}
Proposition~\ref{prop:quadratic-centered}, Proposition~\ref{prop:krm-fixed-degree} and Proposition~\ref{prop:tail} comprise the key technical components of the proof, and are formally stated and proved in Sections~\ref{sec:quadratic-hermite},~\ref{sec:cubic-hermite} and~\ref{sec:tail-gap} respectively.

\paragraph{Error assembly and rate simplification.}
Note that the coefficients $\beta_1,\beta_2,\beta_3$ in \eqref{eq:conditional-hermite-decomposition} are bounded
uniformly since under our assumption, $t$ stays in a fixed compact interval. Minkowski's inequality
and the preceding component bounds give
\begin{equation}
\begin{aligned}
\left(\E\left[\frac d{n^2}\opnorm{\tilde S-S}^2\right]\right)^{1/2}
\le C_t\Bigg(&\sqrt{\frac d{\tau_2}}\frac{\mathcal B_n}{n}
+\frac{\sqrt d}{\stablerank}
+\sqrt{\frac{d(1+\eta_n^3)\log n}{n}}\\
&+\sqrt{d\kappa_3}+\frac{\sqrt d}{\effectivedim}
+\frac{\sqrt d\,\eta_n}{n}\Bigg),
\end{aligned}
\label{eq:oracle-scaled-bound}
\end{equation}
where we used that $(1+\eta_n)^2\le C(1+\eta_n^3)$. Now, we denote as shorthand $a_n=\sqrt{d\log n/n}$ and
$b_n=\sqrt d/\stablerank$. We will now show that the right-hand side of~\eqref{eq:oracle-scaled-bound} is
at most $C_t(a_n+b_n)$.
First, we have
\begin{equation*}
\sqrt{\frac{d(1+\eta_n^3)\log n}{n}}
\le C a_n\left(1+
\left(\frac{\log n}{\effectivedim}\right)^{3/4}
+\left(\frac{\log n}{\stablerank}\right)^{3/2}\right).
\end{equation*}
The second term is at most $a_n$ if $\effectivedim\ge\log n$.
If instead $\effectivedim \le \log n$, we can use the fact that $\stablerank\le\effectivedim$ to show that
its ratio to $b_n$ is at most $(\log n)^{3/2}/\sqrt n$.
Similarly, the third term is
at most $a_n$ if $\stablerank\ge\log n$.
If instead $\stablerank \le \log n$, we can use the fact that $\stablerank \ge 1$ to show that its ratio to
$b_n$ is at most $(\log n)^2/\sqrt n$. Clearly, these ratios are bounded
for all sufficiently large $n$.
Next, the bounds $\effectivedim \ge \stablerank \ge 1$ imply that $\kappa_3\le2/\stablerank^2$ and thus
$\sqrt{d\kappa_3}+\sqrt d/\effectivedim\le Cb_n$.
Moreover, $\tau_1/\sqrt{\tau_2}\le\sqrt d$ and
$d\le n\log n$ imply that
\[
\sqrt{\frac d{\tau_2}}\frac{\mathcal B_n}{n}
\le\frac dn+\frac{\sqrt{d\log n}}n
+\frac{\sqrt d\log n}{n\sqrt{\stablerank}}
\le C a_n.
\]
Finally, since $\effectivedim \geq \stablerank \geq 1$, we have $\eta_n\le\sqrt{\log n}+\log n$, which in turn implies that
$\sqrt d\,\eta_n/n\le Ca_n$. 
We have shown the desired inequality
\begin{equation}
\left(\E\left[\frac d{n^2}\opnorm{\tilde S-S}^2\right]\right)^{1/2}
\le C_t(a_n + b_n) := C_t\left(\sqrt{\frac{d\log n}{n}}
+\frac{\sqrt d}{\stablerank}\right).
\label{eq:oracle-simplified-bound}
\end{equation}

\paragraph{Incorporating the edge density estimation error.}

The last step in proving Proposition~\ref{prop:operator-denoising} is to upper bound the estimation error arising from estimating the edge density, given by $\opnorm{\tilde S - \hat S}$.
First, note that we can write this error matrix as
\begin{equation*}
\hat S-\tilde S
=
\left(\frac{\beta_1}{\hat\beta_1}-1\right)\tilde S
+\frac{\delta_{\mathrm G}}{\hat\beta_1}H,
\end{equation*}
where we define $\delta_{\mathrm G} := \hat p - p_{\mathrm G}$.
Recall that we have assumed that the normalized threshold $t$ is contained in a fixed compact interval.
Let $K$ denote this interval, and choose
\(\varepsilon_K>0\) such that
\[
\bar\Phi(s)\in[2\varepsilon_K,1-2\varepsilon_K] \;\forall s\in K.
\]
Then, we define a ``good event" on edge density estimation as \(\mathcal E_n:=\{|\delta_{\mathrm G}|\le\varepsilon_K\}\).
Note that on this event, we have $\hat p \in [\varepsilon_K,1 - \varepsilon_K]$.
(Indeed, since $p_{\mathrm G}=\bar\Phi(t)$ and $t\in K$, our choice of
$\varepsilon_K$ ensures that
$p_{\mathrm G}\in[2\varepsilon_K,1-2\varepsilon_K]$. On $\mathcal E_n$,
we have $|\hat p-p_{\mathrm G}|\le\varepsilon_K$, and hence
$\hat p
\in
[p_{\mathrm G}-\varepsilon_K,p_{\mathrm G}+\varepsilon_K]
\subseteq
[\varepsilon_K,1-\varepsilon_K].$)
We apply Lemma~\ref{lem:edge-density-concentration} with the choice $s = \varepsilon_K/2$.
Equations~\eqref{eq:technical-effective-dimension-ordering} and~\eqref{eq:rate-bounded-by-constant} give
\[
\frac1{\effectivedim}
\le \frac1{\stablerank}
\le \frac d{\stablerank^2}
\le c_t.
\]
Choosing $c_t$ sufficiently small therefore ensures that
$C_t/\effectivedim\le\varepsilon_K/2$.  Thus
Lemma~\ref{lem:edge-density-concentration} gives
\(\Pp(\mathcal E_n^c)\le 2e^{-c_Kn}\) for all sufficiently large \(n\),
where \(c_K>0\) depends only on \(K\).
We will now bound the quantity $|\beta_1/\hat \beta_1 - 1|$ under the good event $\mathcal E_n$.
We have 
\[
|\beta_1 - \hat \beta_1| = |\varphi(\bar{\Phi}^{-1}(p_G)) - \varphi(\bar{\Phi}^{-1}(\hat p))|.
\]
Now, the \(b(u)=\varphi(\bar\Phi^{-1}(u))\) is continuously differentiable and bounded
away from zero on the interval $[\epsilon_K, 1 - \epsilon_K]$.  The mean-value theorem therefore gives
\[
\left|\frac{\beta_1}{\hat\beta_1}-1\right|
\le C_t|\delta_{\mathrm G}|.
\]
Since \(\beta_1\) is bounded away from zero and
we have the deterministic bound \(\opnorm{A+p_{\mathrm G}I_n}\le n\), we have the deterministic bound
\(\opnorm{\tilde S}\le C_tn\).  Therefore, on the good event \(\mathcal E_n\), we have
\[
\frac{\sqrt d}{n}\opnorm{\hat S-\tilde S}
\le C_t\sqrt d\,|\delta_{\mathrm G}|.
\]
Lemma~\ref{lem:edge-density-concentration} consequently gives
\begin{equation}\label{eq:edge-density-error-goodevent}
\left(
\E\left[
\left(\frac{\sqrt d}{n}
\opnorm{\hat S-\tilde S}\right)^2
\mathbbm{1}_{\mathcal E_n}
\right]
\right)^{1/2}
\le
C_t\left(\sqrt{\frac dn}+\frac{\sqrt d}{\effectivedim}\right).
\end{equation}
It remains to bound the complementary quantity $\E\left[\left(\frac{\sqrt d}{n}\opnorm{\hat S-\tilde S}\right)^2\mathbbm{1}_{\mathcal E_n^c}\right]$. 
For this, we provide a worst-case bound on $\opnorm{\hat S}$.
Recall that we applied a clipping operation to the edge density estimate $\hat p$ which ensures that $n^{-2} \leq \hat p \leq 1 - n^{-2}$.
Moreover, the function \(b(u)\) defined above is symmetric about
\(u=1/2\), and we claim that
\(b(u)\ge c\min\{u,1-u\}\).
To see this, suppose first that \(u\le1/2\) and set
\(z=\bar\Phi^{-1}(u)\ge0\).  
Note that $b(u) = \varphi(z)$.
If \(z\ge1\), then we have
\[
u=\int_z^\infty\varphi(x)\,dx
\le\frac1z\int_z^\infty x\varphi(x)\,dx
=\frac{\varphi(z)}z
\le\varphi(z),
\]
where $z \ge 1$ was used in the last inequality.
If \(0\le z\le1\), then \(u\le1/2\le C\varphi(z)\).  Thus
\(b(u)=\varphi(z)\ge cu\) for \(u\le1/2\), and symmetry with respect to $(u \le 1/2, u > 1/2)$ gives the claim.

From the above reasoning, we have
\(\hat\beta_1 = b(\hat p) \ge cn^{-2}\).  Noting again that
\(\opnorm{A+\hat pI_n}\le n\) gives us the worst case bound
\[
\opnorm{\hat S}\le Cn^3,
\]
Recalling that we have the worst case bound \(\opnorm{\tilde S}\le C_tn\), we use Minkowski's inequality to obtain
\begin{equation}\label{eq:edge-density-error-badevent}
\left(
\E\left[
\left(\frac{\sqrt d}{n}
\opnorm{\hat S-\tilde S}\right)^2
\mathbbm{1}_{\mathcal E_n^c}
\right]
\right)^{1/2}
\le C\sqrt d\,n^2e^{-c_Kn/2}.
\end{equation}
Finally, the contribution from the bad event (the RHS of Equation~\eqref{eq:edge-density-error-badevent}) is negligible compared to the contribution from the good event (the RHS of Equation~\eqref{eq:edge-density-error-goodevent}).
This is because $n^2 e^{-c_K n/2} = o(1/\sqrt{n})$ for large enough $n$.
Ultimately, we obtain the following bound for the edge density estimation error:
\begin{align}\label{eq:edge-density-error}
\left(
\E\left[
\left(\frac{\sqrt d}{n}
\opnorm{\hat S-\tilde S}\right)^2
\right]
\right)^{1/2}
\le
C_t\left(\sqrt{\frac dn}+\frac{\sqrt d}{\effectivedim}\right).
\end{align}
Combining Equations~\eqref{eq:oracle-simplified-bound} and~\eqref{eq:edge-density-error} completes the proof of the proposition.
\end{proof}

\subsection{Frobenius recovery from spectral truncation}

In this section, we convert our scaled operator norm error bound from Proposition~\ref{prop:operator-denoising} into a Frobenius norm squared error bound between the truncated estimator $\hat S_d$ and the centered estimand $S$.

\begin{proposition}
\label{prop:centered-frobenius}
Under the conditions of Proposition~\ref{prop:operator-denoising}, for all
sufficiently large \(n\), we have
\begin{equation}
\E\left[\frac1{n^2}\fnorm{\hat S_d-S}^2\right]
\le C_t\left(
\frac{d\log n}{n}
+\frac{d}{\stablerank^2}
\right).
\label{eq:centered-frobenius-recovery}
\end{equation}
\end{proposition}

Note that the order of the RHS in Proposition~\ref{prop:centered-frobenius} is identical to the RHS of Proposition~\ref{prop:operator-denoising} when squared, up to a constant that depends only on the normalized threshold $t$.

\begin{proof}[Proof of Proposition~\ref{prop:centered-frobenius}]
Recall that we defined the truncated estimator as $\hat S_d := \Pi_d^+(\hat S)$.
We state the following lemma and prove it at the end of this section.

\begin{lemma}
\label{lem:op-to-frobenius-truncation}
Let \(M\) be a symmetric positive semidefinite matrix with rank at most \(d\).
Let \(B\) be a symmetric matrix, and let \(B_d=\Pi_d^+(B)\) be obtained by keeping at
most the \(d\) largest positive eigenvalues of \(B\).  Then, we have
\[
\fnorm{B_d-M}
\le
2\sqrt{2d}\,\opnorm{B-M}.
\]
\end{lemma}

We apply Lemma~\ref{lem:op-to-frobenius-truncation} with $B = \hat S$ and $M = S$. Because $S = H XX^T H/\sqrt{\tau_2}$, i.e. it is a double-centered normalized Gram matrix, it is positive semidefinite and its rank is at most $d$ --- therefore, it satisfies the conditions of Lemma~\ref{lem:op-to-frobenius-truncation}.
This gives us
\[
\frac1{n^2}\fnorm{\hat S_d-S}^2
\le
8\frac d{n^2}\opnorm{\hat S-S}^2.
\]

Taking expectations, applying Proposition~\ref{prop:operator-denoising}, and
using the identity \((a+b)^2\le2(a^2+b^2)\) yields the desired statement of Equation~\eqref{eq:centered-frobenius-recovery}.
It remains to prove Lemma~\ref{lem:op-to-frobenius-truncation}, which we do below.

\begin{proof}[Proof of Lemma~\ref{lem:op-to-frobenius-truncation}]
Denote as shorthand \(E:=B-M\).  We first control the operator norm of the truncation residual
by showing that
\begin{equation}
\opnorm{B-B_d}\le \opnorm{E}.
\label{eq:positive-truncation-residual}
\end{equation}
We need to show that $|\lambda_k(B)| \leq \opnorm{E}$ for all $k \geq d+1$.
Consider first the set of positive tail eigenvalues $\mathcal K_+ := \{k>d: \lambda_k(B) > 0\}$.
If $d \geq n$, this set is empty.
If $d < n$, we have $\lambda_k(B) \leq \lambda_{d+1}(B)$ for all $k \in \mathcal K_+$.
Then, Weyl's inequality yields
\[
\lambda_{d+1}(B)
\le
\lambda_{d+1}(M)+\opnorm{E}
=
\opnorm{E},
\]
where the last equality follows because \(M\succeq0\) has rank at most \(d\).  
For the negative eigenvalues, suppose that \(u\) is a unit eigenvector of \(B\) with negative eigenvalue \(-a\), where \(a>0\).
Then, it suffices to show that $a \leq \opnorm{E}$.
We have
\[
u^\top Eu
=
u^\top Bu-u^\top Mu
=
-a-u^\top Mu
\le -a,
\]
so \(a\le \opnorm{E}\). 
Putting together the bounds on the positive and negative eigenvalues, we have shown Equation~\eqref{eq:positive-truncation-residual}.
After this, the triangle inequality gives us
\[
\opnorm{B_d-M}
\le
\opnorm{B_d-B}+\opnorm{B-M}
\le
2\opnorm{E}.
\]
Finally, because we assumed that $\operatorname{rank}(M) \leq d$, we use the subadditivity property of rank to obtain $\operatorname{rank}(B_d-M)\le
\operatorname{rank}(B_d)+\operatorname{rank}(M)\le2d$.  This gives us
\[
\fnorm{B_d-M}
\le
\sqrt{2d}\,\opnorm{B_d-M}
\le
2\sqrt{2d}\,\opnorm{B-M},
\]
which completes the proof of the lemma.
\end{proof}

Since we have proved Lemma~\ref{lem:op-to-frobenius-truncation}, we have completed the proof of Proposition~\ref{prop:centered-frobenius}.
\end{proof}

\subsection{From centered to uncentered recovery}

Finally, we convert the centered recovery bound to the original uncentered recovery error and prove Theorem~\ref{thm:gram-recovery}.
\begin{proof}[Proof of Theorem~\ref{thm:gram-recovery}]
Recall that the estimand is given by $G=XX^\top/\sqrt{\tau_2}$. First, we suppose, for a sufficiently
small fixed constant $c_t>0$, that
\begin{equation}
R_n:=\frac{d\log n}{n}+\frac d{\stablerank^2}\le c_t.
\label{eq:main-rate-is-small}
\end{equation}
Then, Lemma~\ref{lem:centered-uncentered} gives
\begin{equation}
\left(\E\left[\frac1{n^2}\fnorm{S-G}^2\right]\right)^{1/2}
\le C\left(\frac1{\sqrt n}+\frac{\tau_1}{n\sqrt{\tau_2}}\right)
\le C\sqrt{\frac{d\log n}{n}},
\label{eq:rms-centering-error}
\end{equation}
where the last inequality uses $\tau_1/\sqrt{\tau_2}\le\sqrt d$. 
Because of~\eqref{eq:main-rate-is-small}, the conditions of Proposition~\ref{prop:operator-denoising} (and therefore Proposition~\ref{prop:centered-frobenius}) hold.
Therefore, combining~\eqref{eq:rms-centering-error} with Proposition~\ref{prop:centered-frobenius} yields
\begin{equation*}
\E\left[\frac{1}{n^2}\fnorm{\hat S_d-G}^2\right]
\le C_t\left(\frac{d\log n}{n}+\frac d{\stablerank^2}\right).
\end{equation*}
This completes the proof of the theorem.
%
% Next, we cover the other case where $R_n>c_t$. Here, we establish the uniform bound
% $\E\fnorm{\hat S_d}^2/n^2\le C_t$.
% Choose $T$ with $|t|\le T$. For $a=x^\top\Sigma x/\tau_2$,
% $\E a=1$ and $\E a^2=1+2/\effectivedim\le3$, so
% Paley--Zygmund gives $\Pp(a\ge1/2)\ge1/12$.
% Conditional Gaussianity and symmetry imply $p,1-p\ge q_T$, where
% $q_T=\bar\Phi(\sqrt2T)/12>0$.
% Changing one latent vector changes the empirical edge density by at most
% $2/n$, so the bounded-differences inequality yields
% \[
% \Pp\{\hat p\notin[q_T/2,1-q_T/2]\}\le2e^{-c_T n}
% \]
% for sufficiently large $n$. On the complementary event,
% $\hat\beta_1$ is bounded below and
% \[
% \fnorm{\hat S_d}\le\fnorm{\hat S}
% \le\hat\beta_1^{-1}\fnorm{A+\hat p I_n}\le C_T n.
% \]
% On the exceptional event, clipping gives $\hat\beta_1\ge cn^{-2}$
% as in the plug-in argument, and hence $\fnorm{\hat S_d}\le Cn^3$.
% The exponentially small exceptional probability proves the asserted
% uniform bound. Since
% $\E\fnorm{G}^2/n^2=1+1/n+\tau_1^2/(n\tau_2)$, it follows that
% \[
% \E\frac{\fnorm{\hat S_d-G}^2}{n^2}
% \le C_t\left(1+\frac{\tau_1^2}{n\tau_2}\right)
% \le C_t(1+d/n)\le C'_tR_n
% \]
% when $R_n>c_t$. This completes the proof in all regimes.
\end{proof}

\section{Technical lemmas}
\label{sec:technical-ingredients}

In this section, we state and prove the technical lemmas that were used in the proof of Theorem~\ref{thm:gram-recovery}.

\subsection{Bounds on functionals of Gaussian random matrices and vectors}
\label{subsec:external-inputs}

We first provide basic concentration inequalities on functionals of Gaussian random vectors and matrices.

\begin{lemma}
\label{lem:gaussian-concentration-inputs}
Let \(x_1,\ldots,x_n\overset{\mathrm{iid}}{\sim}N(0,\Sigma)\), and let \(X\) be the data matrix with rows given by \(\{x_i^\top\}_{i=1}^n\).  For every
\(u\ge1\), with probability at least \(1-Ce^{-u}\), we have
\begin{align}
\opnorm{X}
&\le C\left(
\sqrt n\,\|\Sigma\|_{\op}^{1/2}
+\sqrt{\tau_1}
+\sqrt{\|\Sigma\|_{\op}u}
\right),
\label{eq:input-gaussian-matrix-bound}\\
\max_{1\le i\le n}\|x_i\|^2
&\le
\tau_1+C\sqrt{\tau_2(\log n+u)}
+C\|\Sigma\|_{\op}(\log n+u).
\label{eq:input-row-norm-bound}
\end{align}
We also have
\begin{align}
\E\max_{1\le i\le n}\|x_i\|^4
&\le C\left(
\tau_1+\sqrt{\tau_2\log n}+\|\Sigma\|_{\op}\log n
\right)^2 .
\label{eq:input-row-fourth-expectation}
\end{align}
\end{lemma}

\begin{proof}
Applying the Hanson--Wright inequality~\cite[Theorem~1.1]{RV13} to
\(\|x_i\|^2\), followed by a union bound over \(i\in[n]\), yields
Equation~\eqref{eq:input-row-norm-bound}. Integrating this tail bound over
\(u\) yields Equation~\eqref{eq:input-row-fourth-expectation}.

We next prove Equation~\eqref{eq:input-gaussian-matrix-bound}.
First, we calculate $\E[\opnorm{X}]$ using the Sudakov-Fernique inequality.
We write \(X=W\Sigma^{1/2}\), where \(W\) has
independent standard Gaussian entries, and recall that we denoted \(\mu=\opnorm{\Sigma}\).  Let
\(g\sim N(0,I_n)\) and \(q\sim N(0,I_d)\) be independent Gaussian vectors of dimension $n$ and $d$ respectively.  For unit vectors
\(a\in\R^n\) and \(b\in\R^d\), consider the centered Gaussian processes
\[
\mathcal X_{a,b}:=a^\top W\Sigma^{1/2}b,
\qquad
\mathcal Y_{a,b}
:=\sqrt\mu\,\ip{g}{a}+\ip{q}{\Sigma^{1/2}b}.
\]
For two index pairs of unit vectors \((a,b)\) and \((a',b')\), a direct computation gives 
\begin{align*}
&\E_{g,q}(\mathcal Y_{a,b}-\mathcal Y_{a',b'})^2
-\E_{W}(\mathcal X_{a,b}-\mathcal X_{a',b'})^2
=2(1-\ip{a}{a'})
\bigl(\mu-\ip{\Sigma^{1/2}b}{\Sigma^{1/2}b'}\bigr)
\ge0.
\end{align*}
The Sudakov--Fernique inequality~\cite[Theorem~7.2.11]{Ver18} therefore gives
\[
\E\opnorm{X}
\le
\sqrt\mu\,\E\norm{g}+\E\norm{\Sigma^{1/2}q}
\le
\sqrt{n\mu}+\sqrt{\tau_1}.
\]
Moreover, the map \(W\mapsto\opnorm{W\Sigma^{1/2}}\) is also
\(\sqrt\mu\)-Lipschitz with respect to the Frobenius norm.
A standard result on concentration of Lipschitz functions of Gaussian vectors or matrices~\cite[Theorem~5.2.2]{Ver18} consequently gives, for every
\(u\ge1\),
\[
\Pp\left\{
\opnorm{X}>\sqrt{n\mu}+\sqrt{\tau_1}+C\sqrt{\mu u}
\right\}
\le Ce^{-u}.
\]
This proves Equation~\eqref{eq:input-gaussian-matrix-bound}.
\end{proof}

We record an auxiliary estimate used later.

\begin{lemma}
\label{lem:scale-correlation-bounds}
Let
\[
a_i=\frac{x_i^\top\Sigma x_i}{\tau_2},
\qquad
\delta_i=a_i-1.
\]
For every fixed \(p\ge1\),
\begin{equation}
\|\delta_i\|_{L^p}
\le C_p\effectivedim^{-1/2},
\qquad
\E\max_{1\le i\le n}|\delta_i|^p
\le C_p\eta_n^p,
\label{eq:delta-explicit-moments}
\end{equation}
In particular,
\begin{equation}
\E\max_i(1+a_i^3)\le C(1+\eta_n^3).
\label{eq:max-ai-cubic}
\end{equation}
\end{lemma}

\begin{proof}
The random variable \(\delta_i = a_i-1\) is the centered quadratic form
\[
\frac{x_i^\top\Sigma x_i-\tr(\Sigma^2)}{\tr(\Sigma^2)}.
\]
Since \(x_i=\Sigma^{1/2}w_i\), this is
\[
\delta_i=\sum_s\frac{\lambda_s^2}{\tau_2}(w_{is}^2-1).
\]
Here \(\max_s\lambda_s^2/\tau_2=1/\stablerank\) and
\(\sum_s\lambda_s^4/\tau_2^2=1/\effectivedim\).
The Hanson--Wright inequality~\cite[Theorem~1.1]{RV13} gives
\[
\Pp\left\{|\delta_i|>
C\left(\sqrt{\frac{u}{\effectivedim}}+\frac{u}{\stablerank}\right)\right\}
\le 2e^{-u},
\qquad u\ge1.
\]
Since \(\effectivedim\le \stablerank^2\), integrating this tail bound gives
\(\|\delta_i\|_{L^p}\le C_p\effectivedim^{-1/2}\) for fixed \(p\).
After a union bound over \(i\le n\), integration gives the explicit maximal
estimate
\[
\E\max_i|\delta_i|^p
\le
C_p\left(\sqrt{\frac{\log n}{\effectivedim}}+\frac{\log n}{\stablerank}\right)^p
=C_p\eta_n^p.
\]
Because \(a_i=1+\delta_i\), for every fixed \(p,m\ge1\),
\[
\E\max_{1\le i\le n}(1+a_i^m)^p
\le C_{p,m}\bigl(1+\eta_n^{mp}\bigr).
\]
Taking \(p=1\) and \(m=3\) proves \eqref{eq:max-ai-cubic}.
\end{proof}

\begin{lemma}
\label{lem:gaussian-quadratic-column-bound}
Let $x_1,\ldots,x_n\overset{\mathrm{iid}}{\sim}N(0,\Sigma)$ where
$\Sigma\succ0$, and let $T_1,\ldots,T_n\in\R^{d\times d}$ be fixed symmetric matrices.
Then
\begin{equation*}
\E_x\max_{1 \le j \le n} \sum_{i=1}^n \left(x_i^\top T_jx_i-\tr(\Sigma T_j)\right)^2
\le Cn\max_{1 \le j \le n} \fnorm{\Sigma^{1/2}T_j\Sigma^{1/2}}^2.
\end{equation*}
\end{lemma}

\begin{proof}
The random variable
\[
Q_{ij}:=x_i^\top T_jx_i-\tr(\Sigma T_j)
\]
is a centered Gaussian quadratic form.
Denote $w_i:=\Sigma^{-1/2}x_i$ and note that $w_i\sim N(0,I_d)$.
Then, applying the Hanson--Wright inequality~\cite[Theorem 1.1]{RV13}
with $A:=\Sigma^{1/2}T_j\Sigma^{1/2}$ and $X:=w_i$, and using
$\opnorm{A}\le\fnorm{A}$, gives, for every $t\ge1$,
\[
\Pp_x\left\{|Q_{ij}|>t\fnorm{\Sigma^{1/2}T_j\Sigma^{1/2}}\right\}
\le 2e^{-ct}.
\]
Integrating this tail bound using
$\E_x|Q_{ij}|^{2p}=2p\int_0^\infty t^{2p-1}\Pp_x\{|Q_{ij}|>t\}\,dt$
gives, for every $p\ge2$,
\[
\left(\E_x|Q_{ij}|^{2p}\right)^{1/p}
\le Cp^2\fnorm{\Sigma^{1/2}T_j\Sigma^{1/2}}^2.
\]
The covariance identity for Gaussian quadratic forms and the preceding bound
with $p=2$ also give
\[
\E_xQ_{ij}^2=2\fnorm{\Sigma^{1/2}T_j\Sigma^{1/2}}^2,
\qquad
\E_xQ_{ij}^4\le C\fnorm{\Sigma^{1/2}T_j\Sigma^{1/2}}^4.
\]
For each fixed $j$, the variables $Q_{ij}^2-\E_xQ_{ij}^2$ are independent
and centered as $i$ varies.  Moreover,
\[
\E_x(Q_{ij}^2-\E_xQ_{ij}^2)^2
\le\E_xQ_{ij}^4,\qquad
\left(\E_x|Q_{ij}^2-\E_xQ_{ij}^2|^p\right)^{1/p}
\le2\left(\E_x|Q_{ij}|^{2p}\right)^{1/p}.
\]
Now, for every integer $p\ge2$, the Rosenthal inequality~\cite[Theorem~15.11]{BLM13} applied to the variables $Q_{ij}^2-\E_xQ_{ij}^2$ gives
\[
\begin{aligned}
&\left(\E_x\left|\sum_i(Q_{ij}^2-\E_xQ_{ij}^2)\right|^p\right)^{1/p}\\
&\quad\le C\sqrt p\left(\sum_i\E_x(Q_{ij}^2-\E_xQ_{ij}^2)^2\right)^{1/2}
+Cp\left(\sum_i\E_x|Q_{ij}^2-\E_xQ_{ij}^2|^p\right)^{1/p}\\
&\quad\le C\left(\sqrt{np}+p^3n^{1/p}\right)
\fnorm{\Sigma^{1/2}T_j\Sigma^{1/2}}^2.
\end{aligned}
\]
Adding the mean $\sum_i\E_xQ_{ij}^2=2n\fnorm{\Sigma^{1/2}T_j\Sigma^{1/2}}^2$
and using the triangle inequality gives
\[
\left(\E_x\left(\sum_iQ_{ij}^2\right)^p\right)^{1/p}
\le C\left(n+\sqrt{np}+p^3n^{1/p}\right)
\fnorm{\Sigma^{1/2}T_j\Sigma^{1/2}}^2.
\]
Finally, choose $p=\max\{2,\lceil\log n\rceil\}$.
Since $n^{1/p}\le C$, $p^3\le Cn$, and $\sqrt{np}\le Cn$, we have
\[
\begin{aligned}
\E_x\max_j\sum_iQ_{ij}^2
&\le\left(\sum_j\E_x\left(\sum_iQ_{ij}^2\right)^p\right)^{1/p}\\
&\le n^{1/p}\max_j\left(\E_x\left(\sum_iQ_{ij}^2\right)^p\right)^{1/p}
\le Cn\max_j\fnorm{\Sigma^{1/2}T_j\Sigma^{1/2}}^2.
\end{aligned}
\]
% This bound does not require independence across $j$.  
Recalling the definition
of $Q_{ij}$ completes the proof.
\end{proof}

\subsection{Kernel matrix bounds}

In this section, we adapt and sharpen the decoupling and non-commutative Khintchine (NCK) inequality approach of~\cite{KRM25} to our setting, where we need to bound quantities of the form $(\E[\opnorm{K}^2])^{1/2}$, where $K$ is a suitable kernel matrix.
While the decoupling step is identical with the operator norm function replaced by the squared operator norm function (see Equation~\eqref{eq:decouple-to-columns}), the NCK step is specialized to an expectation bound on the operator norm of the Gram matrix of independent columns, viewed as a sum of
positive semidefinite matrices~\cite[Theorem~5.1(1)]{Tropp2016}. 
We state and prove this lemma for completeness.

\begin{lemma}
\label{lem:independent-columns}
Let $N\ge2$ and $m\ge1$, let $a_1,\ldots,a_N$ be independent random vectors
in $\R^m$ such that $\E\max_j\norm{a_j}^2<\infty$, and define the random matrix $A=[a_1\ \cdots\ a_N]$.  Then, we have
\begin{equation}
\left(\E\opnorm{A}^2\right)^{1/2}
\le
\sqrt{2} \opnorm{\sum_{j=1}^N\E[a_ja_j^\top]}^{1/2}
+C\sqrt{\log N}
\left(\E\max_{1\le j\le N}\norm{a_j}^2\right)^{1/2}.
\label{eq:independent-column-bound}
\end{equation}
\end{lemma}

\begin{proof}
Note that $\opnorm{A}^2 = \opnorm{AA^T}$, so it suffices to upper bound the operator norm of the Gram matrix $AA^T$.
Denote as shorthand
\[
S:=AA^\top=\sum_{j=1}^N a_ja_j^\top,
\qquad
\Sigma_A=\E S,
\qquad
M=\E\opnorm{S},
\qquad
b=\left(\E\max_{1 \leq j \leq N}\norm{a_j}^2\right)^{1/2}.
\]
Expressed in this notation, the quantity $\sqrt{M}$ is precisely the left-hand side of Equation~\eqref{eq:independent-column-bound}.
Let $a'_1,\ldots,a'_N$ be independent copies of $a_1,\ldots,a_N$,
independent of one another and of the original vectors, and define
$S':=\sum_j a'_j(a'_j)^\top$.  Let
$\varepsilon := \{\varepsilon_1,\ldots,\varepsilon_N\}$ denote independent Rademacher signs,
independent of all the vectors.  Banach-space symmetrization
\cite[Fact~3.1]{Tropp2016}, applied to the independent matrices
$\{a_ja_j^\top\}_{j=1}^N$, and the triangle inequality give
\begin{align}
\E\opnorm{S-\Sigma_A}
\le \E\opnorm{S-S'} 
\le 2\E\left[\E_\varepsilon
\opnorm{\sum_{j=1}^N\varepsilon_j a_ja_j^\top}\right].
\label{eq:gram-symmetrization}
\end{align}
For fixed $a_1,\ldots,a_N$, restrict the matrices to
$\operatorname{span}\{a_1,\ldots,a_N\}$, which has dimension at most $N$.
The Rademacher matrix series inequality
\cite[Theorem~4.1 and equation~(4.2)]{Tropp2016} gives the explicit estimate
\begin{align*}
\E_\varepsilon
\opnorm{\sum_{j=1}^N\varepsilon_j a_ja_j^\top}
&\le \sqrt{1+2\lceil\log N\rceil}
\opnorm{\left(\sum_{j=1}^N(a_ja_j^\top)^2\right)^{1/2}}.
\end{align*}
Since $(aa^\top)^2=\norm{a}^2aa^\top$ and $aa^T \succeq 0$, we have $\sum_{j=1}^N(a_ja_j^\top)^2
\preceq
\left(\max_{1 \leq j \leq N}\norm{a_j}^2\right)S$.
Substituting this into Equation~\eqref{eq:gram-symmetrization} and applying the Cauchy-Schwarz inequality gives us
\begin{align*}
\E\opnorm{S-\Sigma_A}
\le C\sqrt{\log N}
\cdot \E\left[\max_j\norm{a_j}\,\opnorm{S}^{1/2}\right]
\le C\sqrt{\log N}\,b\sqrt M.
\end{align*}
Consequently, we have
\[
M \le \opnorm{\Sigma_A} + \E\opnorm{S - \Sigma_A} \le \opnorm{\Sigma_A}
+C\sqrt{\log N}\,b\sqrt M.
\]
Applying Young's inequality on the term $C \sqrt{\log N} b \sqrt{M}$ yields
\[
\sqrt M
\le \sqrt{2} \opnorm{\Sigma_A}^{1/2}
+C\sqrt{\log N}\,b.
\]
This completes the proof of the lemma.
\end{proof}
In the remainder of this section, we state and prove for completeness the adaptation of the argument underlying~\cite[Theorem 1]{KRM25} to our setting.
We fix $n\ge2$, let $\{x_1,\ldots,x_n,z\}$ be iid with law $\mathcal P$, and for this section all expectations are taken with respect to the law $\mathcal P$. 
Let $k$ be a jointly measurable, symmetric, real-valued kernel satisfying
\begin{equation}
\E [k(x_1,z)^2]<\infty.
\label{eq:L2-kernel-assumption}
\end{equation}
Note that~\cite{KRM25} instead assumed that $\E|k(x_1,z)| < \infty$; Equation~\eqref{eq:L2-kernel-assumption} is a slightly stronger assumption but will continue to hold in our setting.
We define the quantities $h(x) := \E_z[k(x,z)]$ and $\theta := \E[h(x)]$, and define the associated \emph{canonical kernel} by
\begin{equation*}
k_0(x,y)=k(x,y)-h(x)-h(y)+\theta.
\end{equation*}
Note by definition that $\E_x[k_0(x,y)] = \E_y[k_0(x,y)] = 0$ for independent vectors $x,y \sim P$.
With this notation in hand, we state and prove the proposition that adapts the argument of~\cite{KRM25} to our setting.

\begin{proposition}
\label{prop:offdiag-krm}
Define the off-diagonal kernel matrix corresponding to the kernel $k$ by $\Delta_{ij} := k(x_i,x_j)\mathbbm{1}\{i\ne j\}$.
Assuming that $k$ satisfies Equation~\eqref{eq:L2-kernel-assumption}, we have
\begin{equation*}
\left(\E\opnorm{\Delta}^2\right)^{1/2}
\le C\left(
n\sqrt{\E [h(x_1)^2]}
+\sqrt{n\,\E\opnorm{G}}
+\sqrt{\log n}\left(\E [B_n^2]\right)^{1/2}
\right),
\end{equation*}
where $G$ is the correlation matrix with entries given by $G_{ij} := \E_z[k(x_i,z)k(z,x_j)]$ and we define a ``conditional-variance"-type term by $B_n^2:=\E_{z_1,\ldots,z_n}\max_{1\le j\le n}\sum_{i=1}^n(k(z_j,x_i)-\E_xk(z_j,x))^2.$
\end{proposition}

\begin{proof}
It will be convenient to first prove an upper bound in terms of the correlation matrix and conditional variance terms with respect to the canonical kernel $k_0$ defined above.
In particular, we will show the easier statement
\begin{equation}
\left(\E\opnorm{\Delta}^2\right)^{1/2}
\le C\left(
n\sqrt{\E [h(x_1)^2]}
+\sqrt{(n-1)\,\E\opnorm{G_0}}
+\sqrt{\log n}\left(\E [B_{0,n}^2]\right)^{1/2}
\right),
\label{eq:canonical-kernel-bound}
\end{equation}
where we define $(G_0)_{ij} := \E_z[k_0(x_i,z)k_0(z,x_j)]$ and $B_{0,n}^2 := \E_{(z_1,\ldots,z_n)} \max_{1 \leq j \leq n} \sum_{i \neq j} k_0(x_i,z_j)^2$.
Denote $g(x):=h(x)-\theta$.  The off-diagonal Hoeffding decomposition gives us
\[
\Delta=\theta(\one\one^\top-I_n)+P+\Delta_0,
\]
where $P$ and $\Delta_0$ are matrices defined such that $P_{ij}=(g(x_i)+g(x_j))\mathbbm{1}\{i\ne j\}$ and $(\Delta_0)_{ij}=k_0(x_i,x_j)\mathbbm{1}\{i\ne j\}$.
Note that the matrix $\theta (\one \one^\top - I_n)$ is deterministic.
Writing \(g:=(g(x_1),\ldots,g(x_n))^\top\), we have
\[
P=g\one^\top+\one g^\top-2\diag(g),
\]
and hence, pointwise,
\[
\opnorm{P}
\le2\sqrt n\,\norm{g}+2\norm{g}_\infty
\le4\sqrt n\,\norm{g}.
\]
Therefore, we have 
\begin{align}
\left(\opnorm{\theta(\one\one^\top-I_n)}^2\right)^{1/2}
&=(n-1)|\theta|
\le n\sqrt{\E[h(x_1)^2]},\label{eq:first-hoeffding-term}\\
\left(\E\opnorm{P}^2\right)^{1/2}
&\le4\sqrt n\left(\E\sum_{i=1}^n g(x_i)^2\right)^{1/2}
\le4n\sqrt{\E [h(x_1)^2]}.\label{eq:second-hoeffding-term}
\end{align}
It remains to control the canonical kernel matrix $\Delta_0$.
For each ordered pair $i\ne j$, define the matrix-valued U-statistic
\[
f_{ij}(u,v)=k_0(u,v)e_ie_j^\top.
\]
Observe that $\Delta_0 = \sum_{i \neq j} f_{ij}(x_i,x_j)$, and $\E\opnorm{f_{ij}(x_i,x_j)}^2
=\E k_0(x_i,x_j)^2
<\infty$.
Therefore, we can adapt the decoupling inequality from~\cite[Theorem 1]{deLaPena1992} with respect to the squared operator norm functional to obtain
\begin{equation}\label{eq:decouple-to-columns}
\left(\E\opnorm{\Delta_0}^2\right)^{1/2}
\le 8\left(\E\opnorm{A}^2\right)^{1/2},
\end{equation}
where we define $A := \sum_{i \neq j} f_{ij}(x_i,z_j)$ and $z_1,\ldots,z_n$ is an independent iid copy of the sample.

It remains to bound the decoupled matrix $A$. 
Conditioned on the original sample $X$, the columns $\{a_j\}_{j=1}^n$ of $A$ are independent.
Therefore, we can apply Lemma~\ref{lem:independent-columns} with respect to the conditional expectation on $X$.
To do so, we need to evaluate the terms $\sum_{j=1}^n \E_{z_j}[a_j a_j^\top \mid X]$ and $\E \max_{1 \leq j \leq n} \norm{a_j}^2$.
Denoting as shorthand $D_j=I_n-e_je_j^\top$, we have
\[
\sum_{j=1}^n\E_{z_j}[a_ja_j^\top\mid X]
=\sum_{j=1}^nD_jG_0D_j
=(n-2)G_0+\diag(G_0).
\]
Since
$G_0\succeq0$ and $\opnorm{\diag(G_0)}\le\opnorm{G_0}$, we have $\opnorm{\sum_{j=1}^nD_jG_0D_j}
\le (n-1)\opnorm{G_0}$.
Moreover, the entries of $a_j$ are given by $a_{ij} = k_0(x_i,z_j) \mathbbm{1}\{i \neq j\}$.
Recalling the definition of $B_{0,n}^2$ directly gives us
\[
\E_{z_1,\ldots,z_n}\max_j\norm{a_j}^2=B_{0,n}^2.
\]
Therefore, applying Lemma~\ref{lem:independent-columns} yields
\[
\left(
\E_{z_1,\ldots,z_n}\left[\opnorm{A}^2\mid X\right]
\right)^{1/2}
\le
\sqrt{2(n-1)\opnorm{G_0}}
+C\sqrt{\log n} \cdot B_{0,n}.
\]
Using the identity $(a+b)^2 \leq 2(a^2 + b^2)$, taking the outer expectation over the original sample $X$ and using the subadditivity of the square root function gives us
\begin{align}\label{eq:third-hoeffding-term}
\left(\E\opnorm{A}^2\right)^{1/2}
\le
2\sqrt{(n-1)\,\E\opnorm{G_0}}
+C\sqrt{2\log n}\left(\E B_{0,n}^2\right)^{1/2}.
\end{align}
Combining Equations~\eqref{eq:first-hoeffding-term},~\eqref{eq:second-hoeffding-term} and~\eqref{eq:third-hoeffding-term} yields Equation~\eqref{eq:canonical-kernel-bound}.

Finally, as in the proof of~\cite[Theorem 1]{KRM25}, we adapt the right-hand side of Equation~\eqref{eq:canonical-kernel-bound} to express it in terms of the uncentered quantities $G$ and $B_n$ to complete the proof of Proposition~\ref{prop:offdiag-krm}.
A feature space triangle inequality gives
\begin{equation}
\E\opnorm{G_0}
\le 3\E\opnorm{G}+6n\E h(x_1)^2.
\label{eq:G0-to-G}
\end{equation}
Moreover, the identity $k_0(z_j,x_i)
=(k(z_j,x_i)-\E_xk(z_j,x))-g_i$ and the fact that $(a+b)^2 \leq 2(a^2 + b^2)$ gives us
\begin{equation}
B_{0,n}^2\le2B_n^2+2\sum_{i=1}^ng_i^2
\label{eq:B0-to-B}
\end{equation}
pointwise on $X$.
Taking the outer expectation over $X$ on both sides of Equation~\eqref{eq:B0-to-B}, and then using square-root subadditivity, and combining with Equation~\eqref{eq:G0-to-G} yields the terms
\[
C\sqrt{(n-1)\,\E\opnorm{G}}
+Cn\sqrt{\E [h(x_1)^2]}
+C\sqrt{\log n}(\E [B_n^2])^{1/2}
+C\sqrt{n\log n}\sqrt{\E [h(x_1)^2]}.
\]
Since $\log n\le n$ for $n\ge2$, the last term is absorbed by the
$Cn\sqrt{\E h(x_1)^2}$ term.
Substituting these into Equation~\eqref{eq:canonical-kernel-bound} completes the proof of the proposition.
\end{proof}

\begin{remark}
While the proof of Proposition~\ref{prop:offdiag-krm} is directly inspired by and reminiscent of the proof of~\cite[Theorem 1]{KRM25}, a brief comparison between the results is in order.
As already noted, because Proposition~\ref{prop:offdiag-krm} controls the square of the operator norm rather than just the operator norm, it has a slightly stronger requirement --- that the kernel has a finite second moment.
We now compare the terms in the two bounds.
Define
\[
(B'_n)^2
=n\,\E_{z,x_1,\ldots,x_n}
\max_{1\le i\le n}
(k(z,x_i)-\E_xk(z,x))^2.
\]
Apart from a separate diagonal term, the right-hand side in the bound of~\cite[Theorem 1]{KRM25} is comprised of the three terms
\[
n\sqrt{\log n\,\E h(x_1)^2},
\qquad
\sqrt{n\log n\,\E\opnorm{G}},
\qquad
\log n\,B'_n.
\]
By comparison, the right-hand side in the bound of Proposition~\ref{prop:offdiag-krm} is comprised of the three terms
\[
n\sqrt{\E h(x_1)^2},
\qquad
\sqrt{n\,\E\opnorm{G}},
\qquad
\sqrt{\log n}\,(\E B_n^2)^{1/2}.
\]
Proposition~\ref{prop:offdiag-krm} removes the logarithmic factors for the first two terms.
However, the third terms expressed in terms of $B'_n$ and $B_n$ are similar, but not directly comparable.
\end{remark}

\subsection{Quadratic Hermite term}\label{sec:quadratic-hermite}

In this section, we control the quadratic Hermite term, i.e. $\opnorm{H \Delta^{(2)} H}$.
Without double-centering, the quadratic term $\opnorm{\Delta^{(2)}}$ is not sufficiently small.  To see the obstruction, note that we may assume 
\(\Sigma=\diag(\lambda_1,\ldots,\lambda_d)\) without loss of generality by the rotational invariance of the isotropic Gaussian random vector.
Then, we have
\(\sum_{s=1}^d\lambda_s^2=\tau_2\), and we write \(x_i=\Sigma^{1/2}w_i\), where
\(w_i\sim N(0,I_d)\). 
The matrix $\Delta^{(2)}$ has entries given by
\[
\Delta^{(2)}_{ij}
=
\left(\frac{\ip{x_i}{x_j}^2}{\tau_2}-1\right)\mathbbm{1}\{i\ne j\}
=
\left(
\frac1{\tau_2}\sum_{s,t}\lambda_s\lambda_t
(w_i)_s(w_i)_t(w_j)_s(w_j)_t-1
\right)\mathbbm{1}\{i\ne j\}.
\]
Although each off-diagonal entry is \emph{unconditionally} centered, it is not \emph{conditionally}
centered given \(x_i\).  Indeed, we have
\[
\E_{x_j}\left[\Delta^{(2)}_{ij}\,\middle|\,x_i\right]
=
\frac{x_i^\top\Sigma x_i}{\tau_2}-1 =: h_i ,
\qquad j\ne i,
\]
and $h_i$ has variance
\[
\frac{\Var(x_i^\top\Sigma x_i)}{\tau_2^2}
=
\frac{2\tr(\Sigma^4)}{\tr(\Sigma^2)^2}
=
\frac{2}{\effectivedim}.
\]
This conditional mean calculation yields a lower bound on $\opnorm{\Delta^{(2)}}$.  
Conditional on \(x_i\), the \(n-1\) summands in
\((\Delta^{(2)}\one)_i\) are independent and have common mean \(h_i\).
Jensen's inequality therefore gives
\[
\E\!\left[
|(\Delta^{(2)}\one)_i|\,\middle|\,x_i
\right]
\ge (n-1)|h_i|.
\]
Moreover, Gaussian hypercontractivity \cite[Chapter~5]{Janson1997} gives
\[
\left(\E|h_i|^4\right)^{1/4}
\le 3\left(\E h_i^2\right)^{1/2},
\]
and H\"older's inequality gives
\[
\E h_i^2
\le \left(\E|h_i|\right)^{2/3}
     \left(\E|h_i|^4\right)^{1/3}.
\]
Consequently, we have
\[
\E|h_i|
\ge \frac{\left(\E h_i^2\right)^{3/2}}
          {\left(\E|h_i|^4\right)^{1/2}}
\ge \frac19\left(\E h_i^2\right)^{1/2}
=\frac{\sqrt2}{9\sqrt{\effectivedim}}.
\]
Finally, we have
\[
\opnorm{\Delta^{(2)}}
\ge \frac{\|\Delta^{(2)}\one\|}{\sqrt n}
\ge \frac1n\sum_{i=1}^n|(\Delta^{(2)}\one)_i|.
\]
Combining the above displays gives the lower bound
\begin{equation*}
\E\opnorm{\Delta^{(2)}}
\ge \frac{cn}{\sqrt{\effectivedim}}.
\end{equation*}
To compare this lower bound with the signal from the estimand $G := XX^\top/\sqrt{\tau_2}$, recall that in the regime \eqref{eq:main-rate-is-small} the Gram matrix satisfies
\[
\left\|\frac{XX^\top}{\sqrt{\tau_2}}\right\|_{\op}
=O_{\Pp} \left(
\frac{n\mu}{\sqrt{\tau_2}} \right)
=
\frac{n}{\sqrt{\stablerank}}
\]
by Equation~\eqref{eq:input-gaussian-matrix-bound}.
Therefore, the noise and the signal are comparable in magnitude whenever
\(\effectivedim\asymp \stablerank\), which obstructs this direct perturbation argument for the uncentered spectral method.
Note that this is satisfied even in the case where $\Sigma$ is well-conditioned.
We now show that double-centering removes this obstruction and controls the quadratic Hermite term.

\begin{proposition}
\label{prop:quadratic-centered}
Let \(x_i\overset{\mathrm{iid}}{\sim}N(0,\Sigma)\), where
\(\Sigma\succ0\).  Define
\[
\Delta^{(2)}_{ij}
=
\left(\frac{\ip{x_i}{x_j}^2}{\tau_2}-1\right)\mathbbm{1}\{i\ne j\}.
\]
Then, we have
\begin{equation*}
\left(\E\opnorm{H\Delta^{(2)}H}^2\right)^{1/2}
\le C\left(
\frac{n}{\stablerank}
+(1+\eta_n)\sqrt{n\log n}
+\eta_n
\right),
\end{equation*}
where $\eta_n$ is defined in Equation~\eqref{eq:proof-abbreviations}.
\end{proposition}

\begin{proof}
We decompose the quadratic kernel into its canonical part
% , the two rank-one
% contributions generated by its first-order Hoeffding projection, 
and a
diagonal correction.  Set
\[
a_i=\frac{x_i^\top\Sigma x_i}{\tau_2},
\qquad
r_i=a_i-1,
\qquad
R_r=\diag(r_1,\ldots,r_n),
\]
and define the off-diagonal matrix $\bar\Delta^{(2)}$ by
\[
\bar\Delta^{(2)}_{ij}
=
\left(
\frac{\ip{x_i}{x_j}^2}{\tau_2}
-a_i-a_j+1
\right)\mathbbm{1}\{i\ne j\}.
\]
Then, we have
\[
\Delta^{(2)}-\bar\Delta^{(2)}
=
r\one^\top+\one r^\top-2R_r .
\]
Since \(H\one=0\), double-centering eliminates the two rank-one
contributions by the first-order Hoeffding projection. This gives us
\begin{equation}
H\Delta^{(2)}H
=
H\bar\Delta^{(2)}H-2HR_rH .
\label{eq:quadratic-hoeffding-term-removal}
\end{equation}
It remains only to control the canonical matrix $\bar \Delta^{(2)}$ and the diagonal correction $R_r$.

For the diagonal term in \eqref{eq:quadratic-hoeffding-term-removal}, since \(H\) is an orthogonal projection, we have
\[
\opnorm{HR_rH}
\le \opnorm{R_r}
=\max_{i\le n}|a_i-1|.
\]
Equation~\eqref{eq:delta-explicit-moments} with \(p=2\) therefore gives
\[
\left(\E\opnorm{HR_rH}^2\right)^{1/2}
\le
\left(\E\max_{i\le n}|a_i-1|^2\right)^{1/2}
\le C\eta_n.
\]

The term $H\bar\Delta^{(2)}H$ is controlled by Lemma~\ref{lem:quadratic-canonical}. Combining Equation~\eqref{eq:quadratic-hoeffding-term-removal} with
Lemma~\ref{lem:quadratic-canonical} and the above bound proves the proposition.
\end{proof}

% Let us now control the canonical quadratic matrix.

\begin{lemma}
\label{lem:quadratic-canonical}
Let \(x_1, \dots, x_n \overset{\mathrm{iid}}{\sim}N(0,\Sigma)\) where
\(\Sigma\succ0\).  For a vector \(x\in\R^d\), define
\(\Omega_x:=xx^\top-\Sigma\), and let
\[
\bar\Delta^{(2)}_{ij}
=\frac{\tr(\Omega_{x_i}\Omega_{x_j})}{\tau_2}\mathbbm{1}\{i\ne j\}
=\left(\frac{\ip{x_i}{x_j}^2}{\tau_2}-a_i-a_j+1\right)\mathbbm{1}\{i\ne j\}.
\]
Then, we have
\[
\left(\E\opnorm{\bar\Delta^{(2)}}^2\right)^{1/2}
\le C\left(
\frac{n}{\stablerank}
+(1+\eta_n)\sqrt{n\log n}
\right),
\]
where $\eta_n$ is defined in Equation~\eqref{eq:proof-abbreviations}.
\end{lemma}

\begin{proof}
We apply Proposition~\ref{prop:offdiag-krm} to the symmetric kernel
\[
\bar k(x,z):=\frac{\tr(\Omega_x\Omega_z)}{\tau_2}.
\]
Its first-order Hoeffding projection vanishes because
\[
\E_z\bar k(x,z)
=
\frac{1}{\tau_2}\tr\!\left(\Omega_x\E_z\Omega_z\right)
=0.
\]
To bound the correlation matrix in Proposition~\ref{prop:offdiag-krm}, set
\[
\widetilde\Omega_x:=\Sigma^{1/2}\Omega_x\Sigma^{1/2},
\qquad
u_i:=\Sigma^{1/2}x_i\sim N(0,\Sigma^2),
\]
and let $V\in\R^{n\times d^2}$ have rows
$\operatorname{vec}(\widetilde\Omega_{x_i})^\top
=\operatorname{vec}(u_i u_i^\top-\Sigma^2)^\top$.
The covariance identity for Gaussian quadratic forms, obtained directly from
Isserlis' formula~\cite{Isserlis1918}, gives
\[
\bar G_{ij}:=\E_z[\bar k(x_i,z)\bar k(z,x_j)]
=\frac2{\tau_2^2}\tr(\Omega_{x_i}\Sigma\Omega_{x_j}\Sigma),
\qquad
\bar G=\frac2{\tau_2^2}VV^\top.
\]
We denote as shorthand
\[
M:=\E\opnorm{VV^\top},
\qquad
\mathcal D_n:=\tau_2+\sqrt{\tau_4\log n}+\mu^2\log n
=\tau_2(1+\eta_n).
\]
Proposition~\ref{prop:offdiag-krm} therefore gives
\begin{equation}
\left(\E\opnorm{\bar\Delta^{(2)}}^2\right)^{1/2}
\le C\left(
\frac{\sqrt{nM}}{\tau_2}
+\sqrt{\log n}(\E B_n^2)^{1/2}
\right),
\label{eq:quadratic-canonical-reduction}
\end{equation}
where
\[
B_n^2
=\E_{z_1,\ldots,z_n}\max_{1\le j\le n}
\sum_{i=1}^n\bar k(x_i,z_j)^2
\]
and $z_1,\ldots,z_n$ are i.i.d.\ $N(0,\Sigma)$ and independent of
$x_1,\ldots,x_n$.
It remains to control $M$ and $\E B_n^2$.  The bound on $M$ will use
Proposition~\ref{prop:offdiag-krm} a second time through a self-bounding argument.

For the term $\E B_n^2$, 
Lemma~\ref{lem:gaussian-concentration-inputs} applied to $u_i$ with covariance
$\Sigma^2$, together with $\tau_4\le\tau_2^2$, gives
\begin{equation}
\E\max_i\fnorm{\widetilde\Omega_{x_i}}^2
\le 2\E\max_i\|u_i\|^4+2\tau_4
\le C\mathcal D_n^2.
\label{eq:quadratic-weighted-diagonal}
\end{equation}
Conditional on $z_1,\ldots,z_n$, we apply
Lemma~\ref{lem:gaussian-quadratic-column-bound} with $T_j=\Omega_{z_j}$.
Taking the outer expectation over $z_1,\ldots,z_n$ and applying
Equation~\eqref{eq:quadratic-weighted-diagonal} gives
\begin{equation}
\E B_n^2
=\frac1{\tau_2^2}\E\max_j\sum_i\tr(\Omega_{x_i}\Omega_{z_j})^2
\le\frac{Cn}{\tau_2^2}\E\max_j\fnorm{\widetilde\Omega_{z_j}}^2
\le\frac{Cn\mathcal D_n^2}{\tau_2^2}.
\label{eq:quadratic-canonical-column-bound}
\end{equation}

We now characterize $M$ through a self-bounding argument.
We define the kernel matrix $K:=VV^\top$ and note that
\[
K_{ij}=\tr(\widetilde\Omega_{x_i}\widetilde\Omega_{x_j}).
\]
We decompose $K$ into its diagonal and off-diagonal component and use the
triangle inequality to obtain
$\opnorm{K}\le\opnorm{\diag(K)}+\opnorm{K-\diag(K)}$.
For the diagonal component, Equation~\eqref{eq:quadratic-weighted-diagonal} gives
\[
\E\opnorm{\diag(K)}
=\E\max_i\fnorm{\widetilde\Omega_{x_i}}^2
\le C\mathcal D_n^2.
\]
To control the off-diagonal component, we apply
Proposition~\ref{prop:offdiag-krm} to the symmetric kernel
\[
\widetilde k(x,z):=\tr(\widetilde\Omega_x\widetilde\Omega_z)
=\tr(\Omega_x\Sigma\Omega_z\Sigma).
\]
Its first-order Hoeffding projection vanishes because
$\E_z\widetilde\Omega_z=0$.
The same covariance identity for Gaussian quadratic forms gives the
correlation matrix
\[
\widetilde G_{ij}
=\E_z[\widetilde k(x_i,z)\widetilde k(z,x_j)]
=2\tr(\widetilde\Omega_{x_i}\Sigma^2
       \widetilde\Omega_{x_j}\Sigma^2).
\]
With columnwise vectorization, we therefore have
\[
\widetilde G=2V(\Sigma^2\otimes\Sigma^2)V^\top,
\qquad
\E\opnorm{\widetilde G}\le2\mu^4M.
\]
The column term in this application of Proposition~\ref{prop:offdiag-krm} is
\[
\widetilde B_n^2
:=\E_{z_1,\ldots,z_n}\max_j\sum_i\widetilde k(x_i,z_j)^2.
\]
Conditional on $z_1,\ldots,z_n$, we apply
Lemma~\ref{lem:gaussian-quadratic-column-bound} with
$T_j=\Sigma\Omega_{z_j}\Sigma$.
Since
\[
\fnorm{\Sigma^{1/2}T_j\Sigma^{1/2}}
=\fnorm{\Sigma\widetilde\Omega_{z_j}\Sigma}
\le\mu^2\fnorm{\widetilde\Omega_{z_j}},
\]
taking the outer expectation and applying
Equation~\eqref{eq:quadratic-weighted-diagonal} gives
\[
\E \widetilde B_n^2\le Cn\mu^4\mathcal D_n^2.
\]
Jensen's inequality, together with Proposition~\ref{prop:offdiag-krm}
and the preceding bounds, gives
\[
\E\opnorm{K-\diag(K)}
\le C\left(
\mu^2\sqrt{nM}+\mu^2\mathcal D_n\sqrt{n\log n}
\right).
\]
Since $M\le\E\opnorm{\diag(K)}+\E\opnorm{K-\diag(K)}$,
incorporating the bound on the diagonal component yields
\[
M\le C\left(
\mathcal D_n^2+\mu^2\sqrt{nM}
+\mu^2\mathcal D_n\sqrt{n\log n}
\right).
\]
Young's inequality gives
\[
C\mu^2\sqrt{nM}\le\frac12M+C'\mu^4n,
\]
which, plugged into the above display, results in the bound
\[
M\le C\left(
n\mu^4+\mathcal D_n^2+\mu^2\mathcal D_n\sqrt{n\log n}
\right).
\]

Finally, substituting the above bound on $M$ and Equation~\eqref{eq:quadratic-canonical-column-bound}
into Equation~\eqref{eq:quadratic-canonical-reduction} yields
\[
\left(\E\opnorm{\bar\Delta^{(2)}}^2\right)^{1/2}
\le\frac C{\tau_2}
\left(n\mu^2+\mathcal D_n\sqrt{n\log n}\right)
=C\left(
\frac n{\stablerank}+(1+\eta_n)\sqrt{n\log n}
\right),
\]
where the mixed term
$\sqrt n\,\mu\cdot\mathcal D_n^{1/2}(n\log n)^{1/4}$ is subsumed in the first inequality.
\end{proof}

\subsection{Cubic Hermite term}\label{sec:cubic-hermite}

In this section, we control the cubic Hermite term $\opnorm{\Delta^{(3)}}$.
\begin{proposition}
\label{prop:krm-fixed-degree}
Define the off-diagonal cubic Hermite matrix
\[
\Delta^{(3)}_{ij}
=
\He_3\!\left(\frac{\ip{x_i}{x_j}}{\sqrt{\tau_2}}\right)\mathbbm{1}\{i\ne j\}.
\]
Recall that we defined \(\kappa_3\) in Equation~\eqref{eq:proof-abbreviations}. Then, we have
\begin{equation}
\left(\E\opnorm{\Delta^{(3)}}^2\right)^{1/2}
\le C_3\left(\sqrt{n(1+\eta_n^3)\log n}+n\sqrt{\kappa_3}\right).
\label{eq:krm-hermite}
\end{equation}
\end{proposition}

\begin{proof}
We apply Proposition~\ref{prop:offdiag-krm} to the degree-3 Hermite kernel
\[
k(x,y)=\He_3\!\left(\frac{\ip{x}{y}}{\sqrt{\tau_2}}\right).
\]
For fixed \(x\), the variable \(\ip{x}{z}/\sqrt{\tau_2}\) is a zero-mean Gaussian.  Since \(\He_3(u)=u^3-3u\), we have
\[
\E_z k(x,z)=0,
\]
so the conditional mean term in Proposition~\ref{prop:offdiag-krm}
vanishes.   
Define the correlation matrix in that proposition by
\begin{equation}
G_{ij}
:=
\E_z\left[
\He_3\!\left(\frac{\ip{x_i}{z}}{\sqrt{\tau_2}}\right)
\He_3\!\left(\frac{\ip{z}{x_j}}{\sqrt{\tau_2}}\right)
\right].
\label{eq:def-G-cubic}
\end{equation}
Lemma~\ref{lem:cubic-correlation-matrix} below gives
\[
\E\opnorm{G}
\le C\left(n\kappa_3+(1+\eta_n^3)\log n\right).
\]
It remains to control the $B_n$ term in Proposition~\ref{prop:offdiag-krm}.  Let
\(z_1,\ldots,z_n\) be iid copies of
\(z\), independent of \(x_1,\ldots,x_n\).  In the present setting,
\[
B_n^2
=
\E_{z_1,\ldots,z_n}
\max_{1\le j\le n}\sum_{i=1}^n
\left(
\He_3\!\left(\frac{\ip{z_j}{x_i}}{\sqrt{\tau_2}}\right)
-
\E_x\He_3\!\left(\frac{\ip{z_j}{x}}{\sqrt{\tau_2}}\right)
\right)^2.
\]
Conditional on
\(z_1,\ldots,z_n\), each \(\ip{z_j}{x_i}/\sqrt{\tau_2}\) is Gaussian with variance
\(a_{z_j}=z_j^\top\Sigma z_j/\tau_2\).  Since \(|\He_3(u)|\le C(1+|u|^3)\),
the column sum estimate in 
Lemma~\ref{lem:cubic-column-norm} gives
\[
\E B_n^2
\le Cn(1+\eta_n^3).
\]
Substituting the correlation and column estimates into
Proposition~\ref{prop:offdiag-krm} yields
\[
\left(\E\opnorm{\Delta^{(3)}}^2\right)^{1/2}
\le
C\left(
\sqrt{n\left(n\kappa_3+(1+\eta_n^3)\log n\right)}
+\sqrt{n(1+\eta_n^3)\log n}
\right),
\]
which is precisely Equation~\eqref{eq:krm-hermite}.  
This completes the proof of the proposition.
\end{proof}

It remains to prove Lemmas~\ref{lem:cubic-correlation-matrix} and~\ref{lem:cubic-column-norm}, which we do below.
\begin{lemma}
\label{lem:cubic-correlation-matrix}
For an independent \(z\sim N(0,\Sigma)\), define $G \in \R^{n \times n}$ by Equation~\eqref{eq:def-G-cubic}.
Then, we have
\[
\E\opnorm{G}
\le
C\left(n\kappa_3+(1+\eta_n^3)\log n\right).
\]
\end{lemma}

\begin{proof}
First, we set
\[
u_i=\frac{\Sigma^{1/2}x_i}{\sqrt{\tau_2}},
\qquad
\delta_i=\norm{u_i}^2-1
=\frac{x_i^\top\Sigma x_i}{\tau_2}-1.
\]
Then, applying~\cite[Lemma 5]{KRM25} with $\ell = \ell' = 3$ and expressing it in the above notation yields
\[
G_{ij}
=
6\ip{u_i}{u_j}^3
+9\delta_i\delta_j\ip{u_i}{u_j}.
\]
Introduce the feature maps
\[
\phi(u)=\sqrt6\,u^{\otimes3},
\qquad
\psi(u)=3(\norm{u}^2-1)u.
\]
Both maps have mean zero by Gaussian symmetry, and
\begin{equation}
G_{ij}
=\ip{\phi(u_i)}{\phi(u_j)}
+\ip{\psi(u_i)}{\psi(u_j)}.
\label{eq:cubic-correlation-feature-decomposition}
\end{equation}
We next bound the population covariance operators of the two feature maps given by $\phi$ and $\psi$.  Let
$u$ be a generic copy of the $\{u_i\}_{i=1}^n$.  For a symmetric third-order tensor $T$, define
\[
(c_T)_k=\frac1{\tau_2}\sum_{a,b=1}^dT_{abk}(\Sigma^2)_{ab}.
\]
Define the third-order Wick tensor \cite[Chapter~3]{Janson1997} coordinatewise by
\[
\bigl({:u^{\otimes3}:}\bigr)_{abc}
:=
u_au_bu_c
-\frac{(\Sigma^2)_{ab}}{\tau_2}u_c
-\frac{(\Sigma^2)_{ac}}{\tau_2}u_b
-\frac{(\Sigma^2)_{bc}}{\tau_2}u_a.
\]
Summing this identity against $T$ and using the symmetry of $T$ gives
\[
\ip{T}{u^{\otimes3}}
=\ip{T}{{:u^{\otimes3}:}}+3\ip{c_T}{u},
\]
where \(\ip{T}{{:u^{\otimes3}:}}\) is a third-order Wick polynomial and is therefore orthogonal to the linear term \(\ip{c_T}{u}\), which is a first-order Wick polynomial; see \cite[Theorem~3.9]{Janson1997}. 
By Isserlis' formula and the symmetry of \(T\),
\[
\begin{aligned}
\E\ip{T}{{:u^{\otimes3}:}}^2
&=
\frac6{\tau_2^3}
\sum_{a,b,c,p,q,r=1}^d
T_{abc}T_{pqr}
(\Sigma^2)_{ap}(\Sigma^2)_{bq}(\Sigma^2)_{cr}\\
&\le
C\opnorm{\frac{\Sigma^2}{\tau_2}}^3\fnorm{T}^2
=
\frac{C}{\stablerank^3}\fnorm{T}^2.
\end{aligned}
\]
For the linear term, the Cauchy--Schwarz inequality gives
\[
\begin{aligned}
\norm{c_T}^2
&=
\frac1{\tau_2^2}
\sum_{k=1}^d
\left(
\sum_{a,b=1}^dT_{abk}(\Sigma^2)_{ab}
\right)^2
\le
\frac{\tau_4}{\tau_2^2}\fnorm{T}^2
=
\frac1{\effectivedim}\fnorm{T}^2.
\end{aligned}
\]
Since \(u\sim N(0,\Sigma^2/\tau_2)\), it follows that
\[
\E\ip{c_T}{u}^2
=
c_T^\top\frac{\Sigma^2}{\tau_2}c_T
\le
\frac{C}{\stablerank\effectivedim}\fnorm{T}^2.
\]
Combining these two bounds with the orthogonal decomposition above and the
definition \(\phi(u)=\sqrt6\,u^{\otimes3}\), we obtain
\[
\E\ip{T}{\phi(u)}^2
\le
C\left(
\frac1{\stablerank^3}
+\frac1{\stablerank\effectivedim}
\right)\fnorm{T}^2
=
C\kappa_3\fnorm{T}^2.
\]
Since $u^{\otimes3}$ is symmetric, the supremum defining the covariance
operator norm may be restricted to symmetric tensors.  It follows that
\begin{equation}
\opnorm{\E[\phi(u)\phi(u)^\top]}
\le C\kappa_3.
\label{eq:cubic-tensor-covariance}
\end{equation}
For the second feature, diagonalize $\Sigma$ with eigenvalues
$\lambda_1,\ldots,\lambda_d$ and write
$u=\Sigma w/\sqrt{\tau_2}$, where $w\sim N(0,I_d)$.  Set
\[
\delta
=
\norm{u}^2-1
=
\sum_{\ell=1}^d
\frac{\lambda_\ell^2}{\tau_2}(w_\ell^2-1).
\]
Since \(\delta\) is even in every coordinate, we have
\[
\E(\psi(u)_k\psi(u)_\ell)=0,
\qquad k\ne\ell,
\]
so \(\E(\psi(u)\psi(u)^\top)\) is diagonal in this basis.  Independence
of the Gaussian coordinates and the identities
\[
\E(g^2-1)^2=2,
\qquad
\E(g^2-1)^2g^2=10
\]
give, for every \(k\),
\[
\begin{aligned}
\E(\delta^2w_k^2)
&=
\sum_{\ell=1}^d
\frac{\lambda_\ell^4}{\tau_2^2}
\E((w_\ell^2-1)^2w_k^2)
=
\frac{10\lambda_k^4}{\tau_2^2}
+\frac2{\tau_2^2}\sum_{\ell\ne k}\lambda_\ell^4
=
\frac{2\tau_4+8\lambda_k^4}{\tau_2^2}.
\end{aligned}
\]
Since \(\psi(u)=3\delta u\), we conclude that
\begin{equation}
\begin{aligned}
\opnorm{\E(\psi(u)\psi(u)^\top)}
&=
9\max_{1\le k\le d}
\frac{\lambda_k^2}{\tau_2}\E(\delta^2w_k^2)\\
&\le
C\left(
\left(\max_k\frac{\lambda_k^2}{\tau_2}\right)
\frac{\tau_4}{\tau_2^2}
+
\left(\max_k\frac{\lambda_k^2}{\tau_2}\right)^3
\right)
=
C\kappa_3.
\end{aligned}
\label{eq:cubic-radial-linear-covariance}
\end{equation}

Let $\Phi$ and $\Psi$ be the matrices whose rows are
$\phi(u_i)^\top$ and $\psi(u_i)^\top$, respectively.  The norms 
of the features satisfy
\[
\norm{\phi(u_i)}^2=6(1+\delta_i)^3,
\qquad
\norm{\psi(u_i)}^2=9\delta_i^2(1+\delta_i).
\]
If $M_n=\max_i|\delta_i|$, then
\[
\max_i\norm{\phi(u_i)}^2+\max_i\norm{\psi(u_i)}^2
\le C(1+M_n^3).
\]
Hence
Lemma~\ref{lem:scale-correlation-bounds} gives
\begin{equation}
\E\max_i\norm{\phi(u_i)}^2
+\E\max_i\norm{\psi(u_i)}^2
\le C(1+\eta_n^3).
\label{eq:cubic-feature-max-norm}
\end{equation}
Apply Lemma~\ref{lem:independent-columns} to the independent columns of
$\Phi^\top$ and $\Psi^\top$.  Equations
\eqref{eq:cubic-tensor-covariance}--\eqref{eq:cubic-feature-max-norm},
followed by $(a+b)^2\le2a^2+2b^2$, give
\[
\E\opnorm{\Phi}^2+\E\opnorm{\Psi}^2
\le C\left(n\kappa_3+(1+\eta_n^3)\log n\right).
\]
Since Equation~\eqref{eq:cubic-correlation-feature-decomposition} gives
$G=\Phi\Phi^\top+\Psi\Psi^\top$, we conclude that
\begin{equation*}
\E\opnorm{G}
\le C\left(n\kappa_3+(1+\eta_n^3)\log n\right).
\end{equation*}
This completes the proof of the lemma.
\end{proof}

\begin{lemma}
\label{lem:cubic-column-norm}
Let \(q:\mathbb R\to\mathbb R\) satisfy
\(|q(u)|\le C_q(1+|u|^3)\).  For iid
\(z_1,\ldots,z_n\sim N(0,\Sigma)\), independent of
the iid sample \(x_1,\ldots,x_n\sim N(0,\Sigma)\), define
\[
V_{ij}
=q\!\left(\frac{\ip{z_j}{x_i}}{\sqrt{\tau_2}}\right)
-\E_x q\!\left(\frac{\ip{z_j}{x}}{\sqrt{\tau_2}}\right).
\]
Then
\[
\E\max_{1\le j\le n}\sum_{i=1}^nV_{ij}^2
\le C_q n(1+\eta_n^3).
\]
\end{lemma}

\begin{proof}
We condition on the decoupling variables, apply Rosenthal's
inequality within each column, and then take a maximum over the columns.
Specifically, condition
on the full vector \(z=(z_1,\ldots,z_n)\), and write
\(a_j=z_j^\top\Sigma z_j/\tau_2\).  For each fixed \(j\), the variables
\(V_{1j},\ldots,V_{nj}\) are then iid and centered.  Gaussian moments and
the cubic growth assumption give, for every integer \(p\ge2\),
\[
\|V_{ij}^2\|_{L^p(x\mid z)}
\le C_qp^3(1+a_j^3),
\qquad
\E_x[V_{ij}^2\mid z]\le C_q(1+a_j^3).
\]
Applying the Rosenthal inequality~\cite[Theorem~15.11]{BLM13} to
\(V_{ij}^2-\E_x[V_{ij}^2\mid z]\) yields
\[
\left\|\sum_{i=1}^n
(V_{ij}^2-\E_x[V_{ij}^2\mid z])
\right\|_{L^p(x\mid z)}
\le
C_q(\sqrt{pn}+n^{1/p}p^4)(1+a_j^3).
\]
Take \(p=\lceil\log n\rceil\).  The elementary bounds
\(n^{1/p}\le C\), \(p^4\le Cn\), and \(\sqrt{pn}\le Cn\), followed by
the \(L^p\) maximal inequality over \(j\), show that
\[
\E_x\left[\max_j\sum_iV_{ij}^2\,\middle|\,z\right]
\le C_qn\max_j(1+a_j^3).
\]
Averaging over \(z\) and applying Equation~\eqref{eq:max-ai-cubic} completes the proof of the lemma.
\end{proof}

\subsection{Higher-order residual}
\label{sec:tail-gap}

In this section, we control the higher-order residual $\opnorm{HT_{\geq 4}H}$.
It is necessary to control this residual which consists of all Hermite components of order four or above altogether, rather than individually bounding each of its Hermite components.
This treatment is necessary 
because the threshold function has square summable Hermite coefficients but
not the absolute summability required by a termwise triangle inequality.

\begin{proposition}
\label{prop:tail}
Assume that \(t\) stays in a fixed compact interval.  Let
\[
q_{\ge4}(u)
=
\mathbbm{1}\{u\ge t\}-p_{\mathrm G}-\sum_{k=1}^3\beta_k\He_k(u),
\qquad
(T_{\ge4})_{ij}
=
q_{\ge4}\!\left(\frac{\ip{x_i}{x_j}}{\sqrt{\tau_2}}\right)\mathbbm{1}\{i\ne j\}.
\]
Then, we have
\begin{equation}
\left(\E\opnorm{HT_{\ge4}H}^2\right)^{1/2}
\le C_t\left(\sqrt{n(1+\eta_n^3)\log n}+\frac{n}{\effectivedim}\right).
\label{eq:tail-bound}
\end{equation}
\end{proposition}

\begin{proof}
We introduce the three quantities that arise when
Proposition~\ref{prop:offdiag-krm} is applied to the residual
\(q_{\ge4}\).  For an independent \(z\sim N(0,\Sigma)\), set
\begin{equation}
m(x)=\E_z \left[q_{\ge4}\!\left(\frac{\ip{x}{z}}{\sqrt{\tau_2}}\right)\right],
\label{eq:tail-conditional-mean-definition}
\end{equation}
and define the conditional correlation matrix $\Gamma$ whose entries are given by
\begin{equation}
\Gamma_{ij}
=
\E_z\left[
q_{\ge4}\!\left(\frac{\ip{x_i}{z}}{\sqrt{\tau_2}}\right)
q_{\ge4}\!\left(\frac{\ip{z}{x_j}}{\sqrt{\tau_2}}\right)
\right].
\label{eq:tail-correlation-definition}
\end{equation}
For iid \(z_1,\ldots,z_n\sim N(0,\Sigma)\), independent of
\(x_1,\ldots,x_n\), set
\[
B_{\ge4,n}^2
=
\E_{z_1,\ldots,z_n}
\max_{1\le j\le n}
\sum_{i=1}^n
\left[
q_{\ge4}\!\left(\frac{\ip{z_j}{x_i}}{\sqrt{\tau_2}}\right)
-
\E_x q_{\ge4}\!\left(\frac{\ip{z_j}{x}}{\sqrt{\tau_2}}\right)
\right]^2 .
\]
By definition, \(q_{\ge4}(u)\) is bounded by \(C_t(1+|u|^3)\).
Therefore, Lemma~\ref{lem:cubic-column-norm} gives
\begin{equation}
\E B_{\ge4,n}^2\le C_tn(1+\eta_n^3).
\label{eq:tail-column-assumption}
\end{equation}

By definition, \(T_{\ge4}\) is the off-diagonal kernel matrix associated with
the symmetric kernel
\(q_{\ge4}(\ip{x}{y}/\sqrt{\tau_2})\).
Since 
$\opnorm{HT_{\ge4}H}\le \opnorm{T_{\ge4}}$, it is enough to bound the operator norm of the off-diagonal kernel matrix. 
The function \(q_{\ge4}\) has at most cubic growth, so all integrability
hypotheses in Proposition~\ref{prop:offdiag-krm} hold for Gaussian inputs.
Applying that proposition to this kernel, and using the definitions of
\(m,\Gamma\), and \(B_{\ge4,n}\), gives
\begin{align}
\left(\E\opnorm{HT_{\ge4}H}^2\right)^{1/2}
\le C\Big(&
n\sqrt{\E [m(x_1)^2]}
+\sqrt{n\,\E\opnorm{\Gamma}}
+\sqrt{\log n}\left(\E [B_{\ge4,n}^2]\right)^{1/2}
\Big).
\label{eq:krm-tail-bound-expanded}
\end{align}
The three terms in Proposition~\ref{prop:offdiag-krm} are
estimated by Lemma~\ref{lem:tail-conditional-mean},
Lemma~\ref{lem:tail-correlation}, and Equation~\eqref{eq:tail-column-assumption}, respectively. Substituting them into Equation~\eqref{eq:krm-tail-bound-expanded} gives
\begin{align*}
\left(\E\opnorm{HT_{\ge4}H}^2\right)^{1/2}
&\le
C_t\left(
\frac{n}{\effectivedim}
+\sqrt{n\left(1+\eta_n^3+\frac n{\effectivedim^2}\right)}
+\sqrt{n(1+\eta_n^3)\log n}
\right)\\
&\le
C_t\left(\sqrt{n(1+\eta_n^3)\log n}+\frac{n}{\effectivedim}\right),
\end{align*}
which is precisely Equation~\eqref{eq:tail-bound}.
This completes the proof of the proposition.
\end{proof}

It remains to prove the  estimates used
above.  The next lemma first describes how the
Hermite coefficients of the whole residual change under a variance
perturbation.  

\begin{lemma}
\label{lem:scaled-tail-coefficients}
Let \(g\sim N(0,1)\), and for \(a>0\) define
\[
Q_a(g)=q_{\ge4}(\sqrt a\,g),
\qquad
\gamma_r(a)=\frac1{r!}\E\big[Q_a(g)\He_r(g)\big],
\qquad r\ge0.
\]
For all \(a>0\), we have
\begin{equation}
\sum_{r\ge4} r!\,\gamma_r(a)^2
\le C_t(1+a^3).
\label{eq:tail-high-coeff-l2}
\end{equation}
Moreover, defining \(\delta:=a-1\), for all \(a>0\), we have
\[
|\gamma_0(a)|+|\gamma_1(a)|
\le C_t|\delta|^2,
\qquad
|\gamma_2(a)|+|\gamma_3(a)|
\le C_t(|\delta|+|\delta|^2).
\]
\end{lemma}

\begin{proof}
We first control the coefficients of order four and above by Parseval's theorem.  Since
\[
q_{\ge4}(u)
=
\mathbbm{1}\{u\ge t\}-p_{\mathrm G}-\beta_1u-\beta_2\He_2(u)-\beta_3\He_3(u),
\]
we have \(|q_{\ge4}(u)|\le C_t(1+|u|^3)\).  Hence, for all \(a>0\),
\[
\sum_{r\ge4} r!\,\gamma_r(a)^2
\le
\E Q_a(g)^2
\le C_t(1+a^3),
\]
where the first inequality in the above display used Parseval's theorem.
This proves Equation~\eqref{eq:tail-high-coeff-l2}.

We next control the four low-order coefficients.  Put \(s_a=t/\sqrt a\).
The standard truncated-normal identity is
\begin{equation*}
\E\big[\mathbbm{1}\{g\ge s\}\He_r(g)\big]
=
\varphi(s)\He_{r-1}(s),\qquad r\ge1.
\end{equation*}
Then, using the Hermite multiplication theorem gives
\[
\He_2(\sqrt a\,g)=a\He_2(g)+(a-1),
\qquad
\He_3(\sqrt a\,g)
=a^{3/2}\He_3(g)+3\sqrt a\,(a-1)\He_1(g).
\]
Putting these identities together gives us the following explicit expressions for the four low-order coefficients:
\begin{align}
\gamma_0(a)
&=
\bar\Phi(s_a)-p_{\mathrm G}-\beta_2(a-1),
\label{eq:gamma-zero-explicit}\\
\gamma_1(a)
&=
\varphi(s_a)-\beta_1\sqrt a
-3\beta_3\sqrt a\,(a-1),
\label{eq:gamma-one-explicit}\\
\gamma_2(a)
&=
\frac{s_a\varphi(s_a)}2-a\beta_2,
\label{eq:gamma-two-explicit}\\
\gamma_3(a)
&=
\frac{(s_a^2-1)\varphi(s_a)}6-a^{3/2}\beta_3.
\label{eq:gamma-three-explicit}
\end{align}
At \(a=1\), these formulas and the fact that
\[
\beta_1=\varphi(t),\qquad
\beta_2=\frac{t\varphi(t)}2,\qquad
3\beta_3=\frac{(t^2-1)\varphi(t)}2
\]
show that \(\gamma_r(1)=0\) for \(0\le r\le3\).  Differentiating
Equations~\eqref{eq:gamma-zero-explicit} and~\eqref{eq:gamma-one-explicit} with respect to $a$ also gives
\(\gamma_0'(1)=\gamma_1'(1)=0\).  The second derivatives of
\(\gamma_0,\gamma_1\) and the first derivatives of \(\gamma_2,\gamma_3\) are
uniformly bounded for \(a\in[1/2,3/2]\) and \(t\) in the fixed compact
interval.  Taylor's theorem therefore yields
\[
|\gamma_0(a)|+|\gamma_1(a)|
\le C_t|a-1|^2,
\qquad
|\gamma_2(a)|+|\gamma_3(a)|
\le C_t|a-1|
\]
whenever \(|a-1|\le1/2\).

It remains to extend the local estimates to all \(a>0\).  If
\(0<a\le1/2\), the functions
\[
\bar\Phi(s),\quad \varphi(s),\quad s\varphi(s),\quad
(s^2-1)\varphi(s)
\]
are uniformly bounded for \(s\in\mathbb R\), and every remaining power of
\(a\) in Equations~\eqref{eq:gamma-zero-explicit}--\eqref{eq:gamma-three-explicit} is
bounded.  Thus all four coefficients are bounded, whereas
\(|a-1|\ge1/2\).  On the other hand, if \(a\ge3/2\), the four formulas give
\[
|\gamma_0(a)|+|\gamma_2(a)|\le C_t(1+a),
\qquad
|\gamma_1(a)|+|\gamma_3(a)|\le C_t(1+a^{3/2}).
\]
On this range, both \(1+a\) and \(1+a^{3/2}\) are bounded by a constant
multiple of \(|a-1|+|a-1|^2\); for \(\gamma_0,\gamma_1\), the same
quantities are bounded by a constant multiple of \(|a-1|^2\).
Combining the three ranges proves the claimed bounds.
\end{proof}

The preceding coefficient estimates will be combined with the following
joint moment bounds for the conditional variance and correlation.

\begin{lemma}
\label{lem:joint-scale-correlation-bounds}
Let
\[
a_i=\frac{x_i^\top\Sigma x_i}{\tau_2},
\qquad
\delta_i=a_i-1,
\qquad
\rho_{ij}
=
\frac{x_i^\top\Sigma x_j}
{(x_i^\top\Sigma x_i)^{1/2}(x_j^\top\Sigma x_j)^{1/2}}
\quad (i\ne j).
\]
For every fixed collection of nonnegative integers
\(\alpha,\beta,q,m\) and every \(i\ne j\), we have
\begin{equation}
\E\left[
(1+a_i^m)(1+a_j^m)
|\delta_i|^\alpha|\delta_j|^\beta|\rho_{ij}|^q
\right]
\le
C_{\alpha,\beta,q,m}\effectivedim^{-(\alpha+\beta+q)/2}.
\label{eq:joint-scale-correlation}
\end{equation}
\end{lemma}

\begin{proof}
Define as shorthand
\[
c_{ij}=\frac{x_i^\top\Sigma x_j}{\tau_2}.
\]
Lemma~\ref{lem:scale-correlation-bounds} gives the fixed-order bounds
\(\|\delta_i\|_{L^s}\le C_s\effectivedim^{-1/2}\) and
\(\|a_i\|_{L^s}\le C_s\).  Conditional on \(x_i\),
\(c_{ij}\) is a zero-mean Gaussian, and its variance is
\(x_i^\top\Sigma^3x_i/\tau_2^2\).  Gaussian hypercontractivity
\cite[Chapter~5]{Janson1997} and
\(\E c_{ij}^2=1/\effectivedim\) therefore give, for every fixed \(s\ge1\),
\[
\|c_{ij}\|_{L^s}\le C_s\effectivedim^{-1/2}.
\]
On the event
\(\mathcal E_{ij}=\{a_i,a_j\ge1/2\}\),
\[
|\rho_{ij}|\le2|c_{ij}|.
\]
Hölder's inequality, with sufficiently high fixed moments, therefore gives
the right side of \eqref{eq:joint-scale-correlation} on \(\mathcal E_{ij}\).
On \(\mathcal E_{ij}^c\), Cauchy's inequality in the \(\Sigma\)-inner product gives
\(|\rho_{ij}|\le1\).
Note that $\mathcal E_{ij}$ is a high-probability event; the one-sided Gaussian quadratic-form tail bound
in \cite[Lemma~1]{LM00} gives
\(\Pp(\mathcal E_{ij}^c)\le C\exp(-c\effectivedim)\).  Another application of Hölder's
inequality absorbs this exponentially small contribution into any prescribed
fixed power of \(\effectivedim^{-1}\).  This proves Equation~\eqref{eq:joint-scale-correlation} and completes the proof of the lemma.
\end{proof}

We next control the conditional mean term.

\begin{lemma}
\label{lem:tail-conditional-mean}
The conditional mean \(m\) defined in
\eqref{eq:tail-conditional-mean-definition} satisfies
\begin{equation*}
\E m(x_1)^2\le C_t\effectivedim^{-2}.
\end{equation*}
\end{lemma}

\begin{proof}
Conditional on \(x\), the variable \(\ip{x}{z}/\sqrt{\tau_2}\) is \(N(0,a)\), where
\(
a=x^\top\Sigma x / \tau_2.
\)
Since \(\He_0=1\), the definition of \(\gamma_0\) in
Lemma~\ref{lem:scaled-tail-coefficients} gives
\[
m(x)=\E_g q_{\ge4}(\sqrt a\,g)=\gamma_0(a).
\]
The same lemma gives the global bound
\(|m(x)|\le C_t(a-1)^2\).  Hence
\[
\E m(x)^2
\le C_t
\E(a-1)^4
\le C_t\effectivedim^{-2},
\]
where the last inequality applies Equation~\eqref{eq:delta-explicit-moments} with \(p=4\).
This completes the proof of the lemma.
\end{proof}

It remains to control the conditional correlation matrix, which we do in the next lemma.

\begin{lemma}
\label{lem:tail-correlation}
The conditional correlation matrix \(\Gamma\) defined in
\eqref{eq:tail-correlation-definition} satisfies
\begin{equation}
\E\opnorm{\Gamma}
\le C_t\left(1+\eta_n^3+\frac n{\effectivedim^2}\right).
\label{eq:tail-correlation-assumption}
\end{equation}
\end{lemma}

\begin{proof}
We treat the diagonal through the cubic growth bound and the off-diagonal
entries through the Hermite expansion.  Recall that
\(a_i=x_i^\top\Sigma x_i/\tau_2\).  Conditional on \(x_i\), we have
$\ip{x_i}{z} /\sqrt{\tau_2}
\sim N(0,a_i).$ 
Since \(|q_{\ge4}(u)|\le C_t(1+|u|^3)\), we have
\begin{align*}
\Gamma_{ii}
&=
\E_z \left[q_{\ge4}\!\left(\frac{\ip{x_i}{z}}{\sqrt{\tau_2}}\right)^2\right]
\le
C_t\left(1+
\E_z\left|\frac{\ip{x_i}{z}}{\sqrt{\tau_2}}\right|^6\right)
\le
C_t(1+a_i^3).
\end{align*}
It follows from Equation~\eqref{eq:max-ai-cubic} that
\[
\E\opnorm{\diag(\Gamma)}
=
\E\max_i\Gamma_{ii}
\le C_t(1+\eta_n^3).
\]
For the case \(i\ne j\), conditional on \((x_i,x_j)\), the pair
\[
\left(\frac{\ip{x_i}{z}}{\sqrt{\tau_2}},
\frac{\ip{x_j}{z}}{\sqrt{\tau_2}}\right)
\]
is centered Gaussian with variances \(a_i,a_j\) and correlation
\(\rho_{ij}\).  Let \(g,h\) be standard Gaussians with
\(\E [gh]=\rho_{ij}\).
The Hermite expansion gives
\begin{equation}
\Gamma_{ij}
=
\E\!\left[q_{\ge4}(\sqrt{a_i}\,g)q_{\ge4}(\sqrt{a_j}\,h)\right]
=
\sum_{r\ge0} r!\,\gamma_r(a_i)\gamma_r(a_j)\rho_{ij}^r .
\label{eq:scaled-tail-correlation-expansion}
\end{equation}
Put $b_i=|\delta_i|+|\delta_i|^2$. 
We note that combining Equation~\eqref{eq:scaled-tail-correlation-expansion} with
Lemma~\ref{lem:scaled-tail-coefficients} gives
\begin{align*}
|\Gamma_{ij}|
\le C_t\Big(
&\delta_i^2\delta_j^2
+\delta_i^2\delta_j^2|\rho_{ij}|
+b_i b_j|\rho_{ij}|^2  \notag
+b_i b_j|\rho_{ij}|^3
+(1+a_i^3)^{1/2}(1+a_j^3)^{1/2}|\rho_{ij}|^4
\Big).
\end{align*}
In the above, the first four terms are the contribution from the indices \(r=0,1,2,3\) in
the sum underlying Equation~\eqref{eq:scaled-tail-correlation-expansion} and follow from Lemma~\ref{lem:scaled-tail-coefficients}.  For the contribution from the indices \(r\ge4\), the Cauchy--Schwarz inequality
and Equation~\eqref{eq:tail-high-coeff-l2} yield
\begin{align*}
\sum_{r\ge4} r!|\gamma_r(a_i)\gamma_r(a_j)|\,|\rho_{ij}|^r
&\le
|\rho_{ij}|^4
\bigg(\sum_{r\ge4}r!\gamma_r(a_i)^2\bigg)^{1/2}
\bigg(\sum_{r\ge4}r!\gamma_r(a_j)^2\bigg)^{1/2}\\
&\le
C_t(1+a_i^3)^{1/2}(1+a_j^3)^{1/2}|\rho_{ij}|^4.
\end{align*}
We now pass from the entrywise estimates to an operator norm bound.  We treat
the off-diagonal part of the rank-one term separately:
\[
\opnorm{(\delta_i^2\delta_j^2\mathbbm{1}\{i\ne j\})_{i,j}}
\le
\opnorm{(\delta_i^2\delta_j^2)_{i,j}}
+\opnorm{\diag(\delta_1^4,\ldots,\delta_n^4)}
\le 2\sum_i\delta_i^4.
\]
By Equation~\eqref{eq:delta-explicit-moments}, we then have $\E[\sum_{i=1}^n \delta_i^4] \leq Cn/\effectivedim^2$. For each remaining term, we use the basic inequality \(\opnorm{M}\le\fnorm{M}\), followed by Jensen's inequality and the joint
moment estimate \eqref{eq:joint-scale-correlation}.  For example,
\[
\E\fnorm{\bigl(b_ib_j|\rho_{ij}|^2\mathbbm{1}\{i\ne j\}\bigr)_{ij}}
\le
\left(
\sum_{i\ne j}\E[b_i^2b_j^2|\rho_{ij}|^4]
\right)^{1/2}
\le C_t\frac n{\effectivedim^2},
\]
because \(b_i^2\le2\delta_i^2+2\delta_i^4\).  The same calculation gives
orders \(n\effectivedim^{-5/2}\), \(n\effectivedim^{-5/2}\), and
\(n\effectivedim^{-2}\), respectively,
for the terms consisting of
\(\delta_i^2\delta_j^2|\rho_{ij}|\),
\(b_ib_j|\rho_{ij}|^3\), and
\((1+a_i^3)^{1/2}(1+a_j^3)^{1/2}|\rho_{ij}|^4\).
Consequently, we have
\[
\E\opnorm{\Gamma-\diag(\Gamma)}
\le
C_t\frac n{\effectivedim^2}.
\]
Combining this with bound on the diagonal term $\E\opnorm{\diag(\Gamma)}$ proves Equation~\eqref{eq:tail-correlation-assumption}, and completes the proof of the lemma.
\end{proof}

\subsection{Centering and edge density estimation}
\label{subsec:proof-centered-uncentered}

Finally, we control two operations used in the recovery argument.  The first lemma
quantifies the difference between centered and uncentered Gram matrices, and
the second controls the error in estimating the edge density.

\begin{lemma}
\label{lem:centered-uncentered}
Let \(x_1,\ldots,x_n\overset{\mathrm{iid}}{\sim}N(0,\Sigma)\), and let \(X\)
be the \(n\times d\) matrix with rows \(\{x_i^\top\}_{i=1}^n\).  Recall that we defined
\(H=I_n-n^{-1}\one\one^\top\).  Then, we have
\begin{equation*}
\left(
\E
\fnorm{\frac{XX^\top-HXX^\top H}{\sqrt{\tau_2}}}^2
\right)^{1/2}
\le
C\left(
\sqrt n+
\frac{\tau_1}{\sqrt{\tau_2}}
\right).
\end{equation*}
\end{lemma}

\begin{proof}[Proof of Lemma~\ref{lem:centered-uncentered}]
We use the empirical-mean decomposition that converts centered recovery into
uncentered recovery.  Set
\[
\bar x=\frac1n\sum_{i=1}^n x_i,
\qquad
X_c=HX.
\]
Then \(X=X_c+\one\bar x^\top\) and \(X_c^\top\one=0\).  Expanding \(XX^\top\)
gives
\begin{equation}
XX^\top-HXX^\top H
=
X_c\bar x\,\one^\top+\one\bar x^\top X_c^\top
+\|\bar x\|^2\one\one^\top .
\label{eq:mean-decomposition}
\end{equation}
The right-hand side has rank at most two.  Indeed, the column and row spaces of
the first two terms lie in the span of \(X_c\bar x\) and \(\one\), while the
last term has rank one in the direction \(\one\).

For a rank-one matrix \(uv^\top\), the Frobenius norm equals
\(\norm{u}\norm{v}\).  The triangle inequality in
\eqref{eq:mean-decomposition} therefore gives the pointwise bound
\begin{equation}
\fnorm{\frac{XX^\top-HXX^\top H}{\sqrt{\tau_2}}}
\le
\frac{2\sqrt n}{\sqrt{\tau_2}}\norm{X_c\bar x}
+\frac n{\sqrt{\tau_2}}\|\bar x\|^2.
\label{eq:xx-hxxh-intermediate-bound}
\end{equation}
The covariance between \(\operatorname{vec}(X_c)\) and \(\bar x\)
vanishes, since $\E[(x_i-\bar x)\bar x^\top]=0$ for each $i$.
Because they are jointly Gaussian, the centered matrix \(X_c\) and the sample
mean \(\bar x\) are therefore independent.  Moreover, we have
\begin{equation*}
\E(X_c^\top X_c)=(n-1)\Sigma,
\qquad
\E(\bar x\bar x^\top)=\frac{\Sigma}{n}.
\end{equation*}
From independence of $X_c$ and $\bar{x}$, we get
\[
\E\norm{X_c\bar x}^2
=\frac{n-1}{n}\tr(\Sigma^2)
=\frac{n-1}{n}\tau_2.
\]
Since \(\bar x\sim N(0,\Sigma/n)\), we also have
\[
\E\|\bar x\|^4
=\frac{\tau_1^2+2\tau_2}{n^2}.
\]
Since \(\tau_1^2\ge\tau_2\), squaring \eqref{eq:xx-hxxh-intermediate-bound}, taking the expectation, and using the above two equations then completes the proof.
\end{proof}

The following result provides guarantees on the edge density estimate.

\begin{lemma}
\label{lem:edge-density-concentration}
Fix a compact range of thresholds \(t\).  Under the anisotropic Gaussian model, we have
\[
\left(\E|\hat p-p_{\mathrm G}|^2\right)^{1/2}
\le C_t\left(n^{-1/2}+\effectivedim^{-1}\right).
\]
Moreover, for all sufficiently large \(n\) and every \(s>0\),
\[
\Pp\left\{|\hat p-p_{\mathrm G}|>\frac{C_t}{\effectivedim}+s\right\}
\le 2\exp\left(-\frac{ns^2}{2}\right).
\]
\end{lemma}

\begin{proof}
We first control the deterministic difference between the true marginal edge
probability and its Gaussian reference value, and then control the
\(U\)-statistic fluctuation around the true marginal probability.  Let
\[
p=\Pp\{\ip{x_1}{x_2}/\sqrt{\tau_2}\ge t\}
\]
denote the true marginal edge probability; recall that
\(p_{\mathrm G}=\bar\Phi(t)\) is its standard-Gaussian reference value.
Conditional on \(x_1\), we have
\[
\frac{\ip{x_1}{x_2}}{\sqrt{\tau_2}}
\sim N(0,a_1) \text{ where }
a_1=\frac{x_1^\top\Sigma x_1}{\tau_2}.
\]
For \(v>0\), define
\[
F_t(v):=\bar\Phi(t/\sqrt v).
\]
The conditional edge probability is therefore
\[
\Pp\left\{
\frac{\ip{x_1}{x_2}}{\sqrt{\tau_2}}\ge t
\,\middle|\,x_1
\right\}
=F_t(a_1).
\]
Averaging over \(x_1\) gives \(p=\E F_t(a_1)\), whereas
\(p_{\mathrm G}=F_t(1)\).  Consequently,
\[
p-p_{\mathrm G}
=
\E(F_t(a_1)-F_t(1)).
\]
The conditional variance \(a_1\) satisfies
\[
\E(a_1-1)=0,\qquad
\E(a_1-1)^2=\frac2\effectivedim,\qquad
\Pp\{a_1<1/2\}\le e^{-c\effectivedim},
\]
where the last inequality follows from a weighted chi-square lower-tail bound.
Let \(\mathcal A=\{a_1\ge1/2\}\).  On \(v\ge1/2\), the first two
derivatives of \(F_t\) are uniformly bounded for \(t\) in the fixed compact
interval.  Taylor's theorem therefore gives, on \(\mathcal A\),
\[
F_t(a_1)-F_t(1)
=
F_t'(1)(a_1-1)+R_1,
\qquad
|R_1|\le C_t(a_1-1)^2.
\]
Although this expansion is restricted to \(\mathcal A\), the linear term is
still negligible.  Indeed, the identity \(\E(a_1-1)=0\) and the fact that $\mathbbm{1}_{\mathcal A} + \mathbbm{1}_{\mathcal A^c} = 1$ pointwise implies that 
\[
\E((a_1-1)\mathbbm{1}_{\mathcal A})
=
-\E((a_1-1)\mathbbm{1}_{\mathcal A^c}).
\]
Since \(0\le a_1<1/2\) on \(\mathcal A^c\), we have
\(|a_1-1|\le1\) there.  Also,
\(|F_t(a_1)-F_t(1)|\le1\).  Splitting the expectation over
\(\mathcal A\) and \(\mathcal A^c\) now yields
\[
\begin{aligned}
|p-p_{\mathrm G}|
&\le
C_t\E(a_1-1)^2+C_t\Pp(\mathcal A^c)
\le
\frac{C_t}{\effectivedim}
+C_t e^{-c\effectivedim}
\le
\frac{C'_t}{\effectivedim}.
\end{aligned}
\]
Recall that the edge density statistic can be written as
\[
\hat p_0
=
\frac{1}{\binom{n}{2}}
\sum_{1\le i<j\le n}A_{ij}.
\]
Since \(\E A_{ij}=p\), this order-two \(U\)-statistic is unbiased:
\(\E\hat p_0=p\).  To control its variance, replace \(x_k\) by an
independent copy and denote the resulting statistic by \(\hat p_0^{(k)}\).
Only the \(n-1\) edge indicators incident to vertex \(k\) can change, and
each changes by at most one.  Therefore,
\[
|\hat p_0-\hat p_0^{(k)}|
\le
\frac{n-1}{\binom{n}{2}}
=
\frac2n.
\]
The Efron--Stein inequality now gives
\[
\Var(\hat p_0)
\le
\frac12\sum_{k=1}^n
\E|\hat p_0-\hat p_0^{(k)}|^2
\le
\frac2n.
\]
Consequently,
\[
\E|\hat p_0-p|^2
=\Var(\hat p_0)
\le \frac{C}{n}.
\]
The same \(2/n\) bounded difference estimate and McDiarmid's inequality give,
for every \(s>0\),
\[
\Pp\{|\hat p_0-p|>s\}
\le 2\exp\left(-\frac{ns^2}{2}\right).
\]

Since
\(\hat p_0-p_{\mathrm G}=(\hat p_0-p)+(p-p_{\mathrm G})\), the preceding
bias and fluctuation bounds control the deviation from \(p_{\mathrm G}\).
Because \(t\) lies in a fixed compact interval, we have \(0 < p_{\mathrm G} < 1\).  Therefore, for large enough \(n\), clipping \(\hat p_0\) to
\([n^{-2},1-n^{-2}]\) cannot increase its distance from \(p_{\mathrm G}\),
which proves both claims for \(\hat p\).
This completes the proof of the lemma.
\end{proof}

\section*{Acknowledgments and declaration of AI usage}
CM is supported in part by NSF CAREER Award 2338062.
VM gratefully acknowledges the support of the NSF (through award CCF-2239151 and award IIS-2212182), an Adobe Data
Science Research Award, and an Amazon Research Award.

The authors identified the central decomposition of the noise term into its quadratic, cubic, and higher-order Hermite components, building on the decoupling approach in VM's prior work~\cite{KRM25}. OpenAI’s GPT-5.5 was used to assist with preliminary calculations, and GPT-5.6 was used to assist in completing proof details and improving the bound in Proposition~\ref{prop:krm-fixed-degree}. GPT-6 was used for creating the plots in the numerical experiments. The authors reviewed all AI-assisted proofs and take full responsibility for the contents of the paper.

\bibliographystyle{alpha}
\bibliography{ref}

\end{document}